\documentclass[12pt]{amsart}

\usepackage[usenames,dvipsnames]{xcolor}

\usepackage{fancyhdr}
\usepackage{amsmath,amsthm,amsfonts,graphicx,amssymb,amscd}
\usepackage[all,cmtip]{xy}
\usepackage{graphicx, overpic, float}
\usepackage[mathscr]{euscript}
\usepackage{hyperref}
\usepackage{enumerate}

\newtheorem{thm}{Theorem}[section]
\newtheorem*{jthm}{Theorem}

\newtheorem{cor}[thm]{Corollary}
\newtheorem{lem}[thm]{Lemma}

\newtheorem{prop}[thm]{Proposition}

\theoremstyle{definition}
\newtheorem{defn}[thm]{Definition}
\newtheorem{definition}[thm]{Definition}

\newcommand{\tec}{Teichm\"uller }

\newcommand{\T}{\mathscr{T}}

\renewcommand{\P}{\mathscr{P}}

\newcommand{\R}{\mathbb{R}}

\newcommand{\C}{\mathbb{C}}
\newcommand{\Cc}{\mathcal{C}}

\newcommand{\M}{\mathscr{M}}

\usepackage{fancybox}

\usepackage{ulem}

\renewcommand{\emph}[1]{{\it #1}}

\usepackage{color}

\title[Quadratic  differentials, half-plane structures, and harmonic maps]{Quadratic differentials, half-plane structures, and harmonic maps to graphs}

\author{Subhojoy Gupta}
\author{Michael Wolf}

\address{Center for Quantum Geometry of Moduli Spaces, Ny Munkegade 118 Bldg. 1530, Aarhus C 8000, Denmark.}

\email{subhojoy@caltech.edu}

\address{Department of Mathematics, Rice University, Houston, Texas, 77005-1892, USA.}
\email{mwolf@rice.edu}

\begin{document}

\setcounter{tocdepth}{4}

\maketitle
\begin{abstract} Let  $(\Sigma,p)$ be a  pointed Riemann surface and $k\geq 1$ an integer. We parametrize the space of meromorphic quadratic differentials on $\Sigma$ with a pole of order $k+2$ at $p$,  having a connected critical graph and an induced metric composed of $k$ Euclidean half-planes. The parameters form a finite-dimensional space $\mathcal{L} \cong \mathbb{R}^{k} \times S^1$ that describe a model singular-flat metric around the puncture with respect to a choice of coordinate chart. This generalizes an important theorem of Strebel, and associates, to each point in  $\mathcal{T}_{g,1} \times \mathcal{L}$, a unique metric spine of the surface that is a  ribbon-graph with $k$ infinite-length edges to $p$. The  proofs study and relate the singular-flat geometry on the surface and the infinite-energy harmonic map from $\Sigma\setminus p$ to a $k$-pronged graph,  whose Hopf differential is that quadratic differential.
\end{abstract}

\section{Introduction}

Holomorphic quadratic differentials on compact Riemann surfaces are central objects in classical \tec theory; for example, they provide coordinates for the \tec space of a compact surface in a number of settings. Such a differential induces a singular-flat metric and a measured foliation on the underlying Riemann surface; the projection map to the leaf-space of this foliation is harmonic. These associated constructions (detailed in \S2)  provide alternative descriptions of these differentials: for example the Hubbard-Masur theorem (\cite{HubbMas}) asserts that a holomorphic quadratic differential on a given Riemann surface is uniquely determined by the measured foliation, and one of us (\cite{Wolf2}) showed that the lift of the differential to the universal cover can be recovered from the equivariant harmonic map to the leaf-space (an $\mathbb{R}$-tree) by taking its Hopf differential. Much of the power and usefulness of these tensors in the theory derive from these equivalent holomorphic, geometric and analytic perspectives. 

For surfaces with punctures, the theory extends to the case of ``integrable" quadratic differentials (with poles of order at most one) and ``Strebel" differentials with poles of order two  (see below), but is incomplete  in general.  In this article we develop these equivalences for certain meromorphic quadratic differentials with higher order poles,  whose corresponding singular-flat metrics of infinite area are analogous to the case of  Strebel differentials. \\

Let $\Sigma$ be a compact Riemann surface of genus $g\geq 1$ and $p$ a marked point. Our study is inspired by the following theorem of Strebel (see Theorem 23.5 in \cite{Streb}) which parameterizes a certain geometrically special subclass of the space of quadratic differentials with poles of order $2$:

\begin{jthm}[Strebel] Let $(\Sigma,p)$ be a pointed Riemann surface as above. Given a constant $c\in \mathbb{R}_+$, there is a unique  meromorphic quadratic differential with a pole of order $2$ at $p$ satisfying the following equivalent properties:
\vspace{.02in}

\begin{itemize}
\item \textnormal{(Analytic)} The residue at the pole is  $-c$, the critical graph connected,  and all the non-critical leaves of the (horizontal) measured foliation are closed. 

\item   \textnormal{(Geometric)} In the induced singular-flat metric, the punctured surface $\Sigma\setminus p$  is a half-infinite Euclidean cylinder of circumference $2\pi c$,  with an interval-exchange identification on its boundary that  yields the critical graph. 

\end{itemize}

\end{jthm}

The connected critical graphs (see \S2 for definitions) in these cases  are also called ``ribbon graphs", and have proved useful in works from Harer-Zagier (\cite{HarZag}) and Penner (\cite{PennDec}) to Kontsevich (\cite{Kont}) and others. 

In this article we generalize Strebel's result to the case of poles of higher order.\\

A \textit{half-plane differential} of order $k\geq 1$  on a pointed Riemann surface $(\Sigma,p)$ is a meromorphic quadratic differential on $\Sigma$ with a single pole at $p$ of order ($k+2$) and  a {connected} critical graph whose complement is a collection of $k$ Euclidean half-planes (what we shall call a \textit{half-plane structure}). Our main result parameterizes the {space} $\mathcal{HP}_k(\Sigma,p)$ of such half-plane differentials in terms of ``local data" at the pole measured with respect to a choice of coordinate chart.

The \textit{collapsing map} to the leaf-space of the horizontal foliation for a half-plane structure defines a harmonic map from $\Sigma\setminus p$  to a metric graph $X_k$ comprising $k$ infinite rays (prongs) meeting at a single vertex $O$. 
The asymptotic behavior of this map on the chosen coordinate chart  shall characterize the differential.  The ``model"  is that of  a \textit{$k$-planar-end}, which is a conformal punctured disk obtained by gluing $k$ Euclidean half-planes by isometries along their boundaries (see Definition \ref{pend}).  As we shall see, the parameters determining this model map are then the ``local data" at the pole.  Our main result is

\begin{thm}\label{main} Let $(\Sigma,p)$ be a pointed Riemann surface of genus $g\geq 1$, and $U\cong \mathbb{D}$ a choice of coordinate chart around $p$.  For any $k\geq 1$ there is a space $\P(k) \cong \mathbb{R}^{k}$ of ``$k$-planar-ends" and a family  $\M(k) \cong \mathbb{R}^k \times S^1$ of harmonic maps from $\mathbb{D}^\ast$ to $X_k$ obtained as their collapsing maps,  such that the following three spaces are homeomorphic:
\vspace{.02in}
\begin{itemize}

\item \textnormal{(Complex-analytic)} $\mathcal{HP}_k(\Sigma,p) = \{\text{half-plane differentials of order } k\newline  \text{ on } (\Sigma,p)\}$,
\vspace{.03in}

\item \textnormal{(Synthetic-geometric)}  $\mathcal{P}_k(\Sigma,p) = \{\text{singular-flat surfaces obtained by} \newline\text{ an identification on the boundaries  of } k \text{ Euclidean half-planes by an}\newline \text{interval-exchange map, such that the resulting surface is } \Sigma \setminus p \}$, 
\vspace{.03in}

\item \textnormal{(Geometric-analytic)} $\mathcal{H}_k(\Sigma,p)= \{\text{harmonic maps from } \Sigma \setminus p \text{ to } X_k \newline \text{ asymptotic}  \text{ (bounded distance) to some } m \in \M(k) \text{ on } U\setminus p  \}$.
\end{itemize}

Here, the maps between the spaces extend standard constructions:  $\textnormal{(C-a)} \to \textnormal{(S-g)}$ assigns the induced metric of the differential,
 $ \textnormal{(S-g)} \to \textnormal{(G-a)}$ assigns the collapsing-map, and 
 $ \textnormal{(G-a)} \to \textnormal{(C-a)}$ assigns the Hopf differential of the harmonic map. 

In particular, given any $m\in \M(k)$, there is a unique $q\in \mathcal{HP}_k(\Sigma,p)$ whose half-plane structure has a collapsing map asymptotic to $m$. 
This  yields a parameterization 
\begin{center}
$\Psi_U: \mathcal{HP}_k(\Sigma,p) \to \M(k) \cong \mathbb{R}^{k} \times S^1$.
\end{center}

The parameters for the space $\M(k)$ are  the \textnormal{(combinatorial)} data of edge-lengths of the metric graph that form the spine for the corresponding $k$-planar-end, and the additional $S^1$ factor   represents the angle at which the planar end sits relative to the chosen coordinate chart.

\end{thm}

\textit{Remark.} For this case of higher order poles,  in contrast with Strebel's theorem stated earlier, specifying the ``residue"  at the pole no longer suffices to uniquely determine the half-plane differential. Instead, one needs to specify the model map, which is a choice of an element in $\M(k)$, at the pole. This involves $k+1$ additional parameters (the combinatorial edge-lengths and angle) that depend on the choice of a co-ordinate chart around the pole: these constitute the \textit{local data} at the pole.\\

\textit{Example.} We provide a quick example to illustrate the  well-known ``standard constructions" alluded to above (see \S2 for details and definitions):  for $k\geq 2$  consider the surface $\mathbb{C}\mathrm{P}^1 = \mathbb{C} \cup \{\infty\}$ with  the quadratic differential $z^{k-2}dz^2$. This has a pole of order $(k+2)$ at $\infty$, and is a half-plane differential  of order $k$: there is a single zero at $0\in \mathbb{C}$, and $k$ critical trajectories into $\infty$ that are rays from the zero at angles of $\{ \frac{2\pi \cdot j}{k} \}_{0\leq j\leq k-1}$.  
This \textit{complex-analytic} object then induces a singular-flat metric  $\lvert z^{k-2}\rvert \lvert dz^2\rvert$ comprising  $k$ half-planes: namely, in this metric  each sector between critical rays is isometric to a Euclidean half-plane  $\left(\{\Im \zeta > 0\}, \lvert d\zeta^2\rvert\right)$ by the map $z\mapsto z^{k/2} =\zeta$. This is the \textit{synthetic-geometric} half-plane structure;  the subset $\mathbb{C} \setminus \{0\}$ is then a $k$-planar end, being a conformal punctured disk built out of half-planes. 
The map $h:\mathbb{C} \to X_k$ defined by collapsing to the leaf-space of the horizontal foliation is defined locally by $\zeta \mapsto \Im\zeta$ where $\zeta$ is the coordinate on each half-plane as above. This is the \textit{geometric-analytic} harmonic map to the $k$-pronged tree,  whose restriction to $\mathbb{C} \setminus \{0\}$ is a ``model map". Finally, its Hopf differential $4\langle h_z, h_z \rangle dz^2$ (see Definition \ref{hopfd}) recovers the differential $z^{k-2}dz^2$ (up to a real constant multiple) returning us to the complex-analytic setting. 

The maps between the spaces in Theorem \ref{main} extend these basic constructions to the setting of  half-plane differentials on punctured Riemann surfaces of genus $g\geq 1$. \\


Descriptions of the spaces $\M(k)$ and $\P(k)$ appear in \S3, and details of the other spaces in Theorem \ref{main} appear in \S4. 
As mentioned earlier, the correspondences in Theorem \ref{main} are known in the context of integrable holomorphic quadratic differentials. Even for the case of differentials with poles of higher order, the correspondence between the ``complex-analytic" and  ``synthetic-geometric" descriptions can be easily derived. 
Such differentials have also been studied from the point of view of extremal problems (see, for example \cite{Kuz}).
 What this present work accomplishes is to identify the correct set of local data, at a pole of higher order, that suffices to describe all possible half-plane differentials on a given pointed Riemann surface.  The difficult part of the theorem is to show that any choice of such local data is (uniquely) achieved. For this we use the ``geometric-analytic" characterization of half-plane differentials, involving harmonic maps to the graph $X_k$.\\

The critical graph of a half-plane differential  in $\mathcal{HP}_k(\Sigma,p)$ forms a metric spine of $\Sigma \setminus p$, with $k$ infinite rays to the puncture. As we vary the local data, and the underlying (marked) pointed Riemann surface in $\mathcal{T}_{g,1}$, these data vary in the ``combinatorial" space $\mathcal{MS}^k_{g,1}$ of all possible such marked metric graphs (see \S6).  Conversely, changing the lengths of the edges in such a metric spine can change the underlying conformal structure as well as the local data at the pole. Hence, as a consequence of Theorem \ref{main}, we establish:

\begin{cor}\label{thm3}
The map $\Phi: \mathcal{T}_{g,1} \times \mathbb{R}^{k} \times S^1 \to \mathcal{MS}^k_{g,1}$ that assigns to a pointed Riemann surface and local data the metric spine of the corresponding  half-plane differential, is a homeomorphism.
\end{cor}

The space on the left hand side can be identified with the total space of half-plane differentials of order $k$, as a subset of the corresponding bundle of meromorphic quadratic differentials over $\mathcal{T}_{g,1}$ (see \S6 for details). The homeomorphism $\Phi$ may be thought of as a correspondence between Teichm\"{u}ller space with ``higher" decoration, and metric spines on the  punctured surface: from that perspective, this corollary generalizes the Penner-Strebel parametrization of decorated Teichm\"{u}ller space (see \cite{PennDec}). \\

The results in this paper generalize in the obvious way to the case of finitely many poles on a compact Riemann surface (and to $\mathbb{C}\mathrm{P}^1$ with more than two poles). The existence of half-plane differentials on a given Riemann surface with prescribed order and residue at the poles was first shown in \cite{Gup25}.  This article provides a different proof of that existence result, and completely answers the question of uniqueness raised in that paper.  One way to interpret our result is that we have completely answered the question: 
\begin{center}
\textit{In how many ways can you glue $k$ Euclidean half-planes by isometries along intervals on their boundaries, to obtain a \textit{given} punctured Riemann surface?}
\end{center}
Half-plane differentials arose in previous work of one of us (\cite{Gup3}) as representing ``limits" of singular-flat surfaces along Teichm\"{u}ller geodesic rays (see Figure 1 for an example).

\begin{figure}
  \centering
  \includegraphics[scale=0.53]{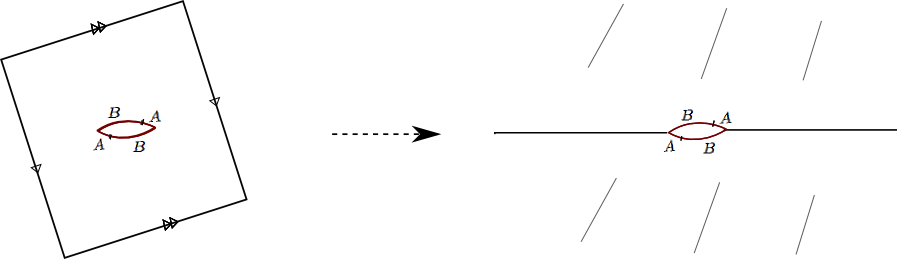}\\
  \caption{\textit{(An example of a half-plane structure.)} The genus-$2$ surface (on the left) is obtained by an interval exchange on a slit of irrational slope on a flat torus.  The corresponding Teichm\"{u}ller ray ``stretches" the singular-flat metric in the vertical direction. A half-plane differential of order $2$ on a punctured torus (on the right) then arises as the Gromov-Hausdorff limit  when basepoints are taken to lie on the critical graph. }
\end{figure}

An \textit{arbitrary} meromorphic quadratic differential with higher order poles may have a disconnected critical graph, which decomposes the surface into ``horizontal strips" and ``spiral" or ``ring" domains in addition to half-planes. In fact a \textit{generic} such differential would have \textit{no} critical trajectories between zeros (\textit{viz.} are saddle-free) and such a decomposition comprises only horizontal strips and half-planes (see also \cite{BridSmi}). 
The present work thus concerns trajectory structures in the  ``most degenerate" case.  In forthcoming work we aim to handle \textit{all} trajectory structures and provide a full generalization of the Hubbard-Masur theorem (\cite{HubbMas}).\\

\textbf{Outline of the paper.} The main result of the paper is the correspondence between the ``complex-analytic" and ``geometric-analytic" spaces in Theorem \ref{main}. The technique of using harmonic maps to graphs to produce prescribed holomorphic quadratic differentials was developed by one of us in \cite{Wolf1}, \cite{Wolf2}. A principal difficulty in the present work is that the harmonic maps in the ``geometric-analytic" characterization have infinite energy. The argument for proving the existence of such a map involves taking a compact exhaustion of the punctured surface, and an appropriate sequence of harmonic maps that converges to it. The convergence is guaranteed by \textit{a priori} energy bounds obtained by comparing the maps with the solutions of a certain ``partially free boundary" problem on an annulus. This occupies an entire section (\S4): we prove that under an additional assumption of symmetry of the annuli the approximating maps (defined on the compact exhaustion) have uniformly controlled behavior.  

A feature of this paper is that we refrain from working in the universal cover: our arguments exploit the ``half-plane structure" on the surface.  In particular, a  crucial challenge is to control the topology of the foliations of the Hopf differentials of the limiting harmonic maps -  for this we use properties of the foliations that are specific to a half-plane structure (see, for example, the Topological Lemma in \S3.2). 

After some preliminaries in \S2, we introduce the notion of planar-ends and their model maps in \S3. After the discussion of the approximating maps in \S4, the proof of Theorem \ref{main} is assembled in  \S5. In \S6 we introduce the space of metric spines  $\mathcal{MS}^k_{g,1}$ and provide the proof of Corollary \ref{thm3}.\\

\textbf{Acknowledgements.} The first author was supported by Danish Research Foundation and the Centre for Quantum Geometry of Moduli Spaces (QGM).  He is grateful to Yair Minsky and Richard Wentworth for helpful conversations, and to Maxime Fortier-Bourque for pointing out \cite{Kuz} and related references. The second author gratefully acknowledges support from the U.S.~National Science Foundation (NSF) through grant DMS-1007383. Both authors gratefully acknowledge support  by NSF grants DMS-1107452, 1107263, 1107367 ``RNMS: GEometric structures And Representation varieties" (the GEAR network) as well as the hospitality of MSRI (Berkeley), where some of this research was concluded. \\

\section{Preliminaries}

In this section we provide some basic background that is relevant to the rest of the paper, including some analytical results concerning harmonic maps that we shall be using later.

\subsection{Quadratic differentials}
In this paper $(\Sigma,p)$ shall denote a Riemann surface of genus $g\geq 1$ and a marked point $p$.

\begin{defn} A \textit{meromorphic quadratic differential} with a pole of order $m\geq 1$ at $p$ is a $(2,0)$-tensor that  is locally of the form $q(z)dz^2$ where $q(z)$ is holomorphic away from $p$, while at $p$ has a pole of order $m$.\end{defn}

\textit{Remark.} By the Riemann-Roch theorem, the meromorphic quadratic differentials on $(\Sigma,p)$ with a pole of order at most $m\geq 1$ is  a complex vector space of dimension $3g-3 + m$.\\

We shall often be thinking of this analytic object in terms of the geometry it induces on the Riemann surface:

\begin{defn}[Singular-flat metric]  The \textit{singular flat metric} induced by a meromorphic quadratic differential is the conformal metric locally defined as $\lvert q(z) \rvert \lvert dz^2\rvert$ (the singularities are at the zeros of the differential). The \textit{horizontal} and \textit{vertical} components of the distance along an arc $\alpha$ shall be the absolute values of the real and imaginary parts, respectively,  of the complex-valued integral $\displaystyle\int_\alpha \sqrt{q}$.  \end{defn}

\begin{defn}[Horizontal  measured foliation] \label{defn: measured foliation}
The \textit{horizontal  foliation} induced by such a differential is the singular foliation on the surface obtained locally by pulling back the horizontal  lines on the $\xi$-plane where the coordinate change $z\mapsto \xi$ transforms the differential to $d\xi^2$. Moreover this is a measured foliation (see \cite{FLP}),  equipped with a measure on transverse arcs coming from the imaginary component distance along such an arc. A similar definition holds for the \textit{vertical foliation} - this time we pull back the vertical lines by the change-of-coordinate map. The foliation around a zero of the differential has a ``branched" structure that we shall refer to as a \textit{prong-singularity}; the order of the zero shall be referred to as the \textit{order} of such a prong-singularity. Around a pole of order $(k+2)$, however, the foliation has a structure with $k$ ``petals" (see Figure 2). 
\end{defn}

\begin{figure}
  \centering
  \includegraphics[scale=0.35]{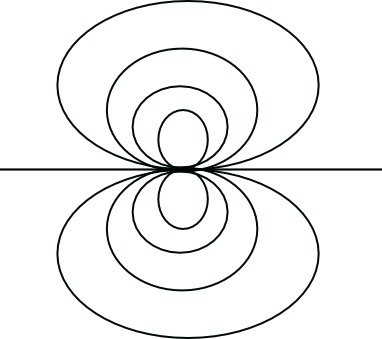}\\
  \caption{The differential $dz^2$ on $\mathbb{C}$ has a pole of order $4$ at infinity, this figure shows the horizontal foliation around the pole.  }
\end{figure}

In classical Teichm\"{u}ller theory, meromorphic quadratic differentials with a pole of order at most $1$ at $p$ appear as the cotangent space to the Teichm\"{u}ller space $\mathcal{T}_{g,1}$ at  $(\Sigma,p)$. These are called \textit{integrable} quadratic differentials as the area form $|\phi| \lvert dz^2\rvert$ is integrable on $\Sigma$, that is,  the total singular flat area is finite. For quadratic differentials with poles of order $m\geq 2$, the induced singular flat metric has \textit{infinite} area. We briefly recall the metric structure around the poles (see \cite{Streb}, \cite{Gup25} for details):
\begin{itemize}
\item an infinite cylinder for $m=2$,
\item a neighborhood of infinity of a $k$-\textit{planar end} with $k=m-2$ for $m>2$ (see Definition \ref{pend}) comprising $k$ Euclidean half-planes. 
\end{itemize}

\subsection{Half-plane differentials}
In this article we shall be concerned with quadratic differentials with poles \textit{of higher order} (that is, of order $m \geq 3$), and in the special case that the critical graph (see the following definition) is connected.

 As mentioned in \S1, we call such a differential a \textit{half-plane differential} of order $k$, where $k=m-2$, and its induced singular-flat geometry a \textit{half-plane structure}. 

\begin{defn}[Critical graph]\label{critg} The \textit{critical graph} of a holomorphic quadratic differential comprises the zeros of the differential as well as the horizontal  trajectories (leaves of the horizontal  foliation) emanating from them. 
\end{defn}

In the case of higher-order poles,  when the critical graph is connected, that graph forms a spine of the punctured surface. It follows from the classification of the trajectory structure of horizontal  foliations (see \cite{Streb}) that the complement of such a graph is then a collection of Euclidean upper half-planes with the horizontal  foliation precisely the horizontal  lines on each half-plane.  (In particular, in this case spiral domains,  ring domains or horizontal  strips do not occur.) Retracting each half-plane to its boundary along the vertical direction then defines a retraction of the punctured surface to the spine. This describes the induced {half-plane structure}.

\subsection{Harmonic maps to graphs} 

Our main analytical tool will be harmonic maps from the punctured Riemann surface to the $k$-pronged graph $X_k$. For the general theory of harmonic maps to singular spaces, we refer to \cite{KorSch1}, \cite{KorSch2} and \cite{GroSch}. For the discussion in this paper, the following definition will suffice:

\begin{defn}\label{har} A harmonic map $h:\Sigma \setminus p \to X_k$ is a continuous and weakly differentiable map with $L^2$-derivatives  that is a critical point of the energy-functional for any compactly supported variation.  (Note that by the above regularity assumptions the energy density is locally integrable.) Alternatively, for this case where the target is a tree, we equivalently require that the pullback of germs of convex functions should be subharmonic (see \S3.4 of \cite{FarWol} and Theorem 3.8 of \cite{DaWen}).
\end{defn}

The relation to quadratic differentials comes from the following construction:

\begin{defn}\label{hopfd} The Hopf differential of a harmonic map $h:\Sigma \setminus p \to X_k$ is the $(2,0)$-part of the pullback of the metric on $X_k$, namely is the quadratic differential locally of the form $4\langle h_z, h_z \rangle dz^2$ (where the scalar product is with respect to a choice of conformal background metric). \end{defn}

The main observation (see \cite{Wolf2}, and  Lemma 1 of \cite{Schoen}) that we use throughout in this paper, is that the Hopf differential is holomorphic. Moreover, for a $1$-dimensional target as in our case,  the harmonic map projects along the leaves of the \textit{vertical} foliation of the Hopf differential. In other words, the level sets of the harmonic map form the vertical foliation of the Hopf differential.

\begin{defn}[Collapsing map]\label{collmap} Given a half-plane differential $q \in \mathcal{HP}_k(\Sigma,p)$ its \textit{collapsing map} $c= c(q):\Sigma\setminus p \to X_k$ collapses the horizontal  foliation of $q$ to its leaf-space, which is a tree with a single vertex $O$ and $k$ infinite-length prongs. The critical graph is mapped to the vertex $O$. Moreover, the collapsing map is harmonic, and its Hopf differential is $-q$ (up to a positive real multiple) . We shall sometimes refer to the horizontal foliation above as the \textit{collapsing foliation}. 
\end{defn}

\textit{Remark.} The extra negative sign above arises because of our choice of realizing the \textit{horizontal} foliation instead of vertical. This agrees with the  convention chosen in \cite{HubbMas}, but is opposite to that of \cite{Wolf2}. \\

We note two further analytical results. First, a standard argument using the chain-rule (see \cite{Jost1} equation (5.1.1) for Riemannian targets and \cite{KorSch1} for tree-targets)  yields:

\begin{lem}\label{dist}
Let  $(T,d)$ be a (locally finite) metric tree with a basepoint $O$.  Then for a harmonic map $h:\Sigma \setminus p  \to T$, the distance function from the basepoint, defined on $\Sigma\setminus p$ as  $d(h(z), O): \Sigma\setminus p \to \R$, is subharmonic. 
\end{lem}

Second, the following convergence criteria follows from a standard argument using, for two-dimensional domains, the Courant-Lebesgue Lemma, a uniform lower bound on the injectivity radius of the domain surface $\Sigma$ and Ascoli-Arzela. (See \cite{Jost1} or \cite{Wolf2} for dimension two; for an analogous result that holds for higher-dimensional domains see \cite{KorSch1}.)

\begin{lem}\label{alem1} Let  $\Sigma$ be a compact Riemann surface (possibly with boundary) and let $(T,d)$ be a (locally finite) metric tree. For $i\geq 1$ let $h_i: \Sigma  \to T$ be a sequence of harmonic maps:\\
(a) with uniformly bounded energy and\\
(b) whose images have uniformly bounded distance from a fixed base-point on $T$.\\
Then there is a convergent subsequence with a limiting harmonic map $h:\Sigma \to T$.
\end{lem}

\subsection{More examples} We have already given a family of examples on the genus-zero surface $\mathbb{C}P^1$ in \S1, following the statement of Theorem \ref{main}.

More generally, a one-parameter family of differentials in  $\mathcal{HP}_k(\mathbb{C}P^1,\infty)$ for an even integer $k\geq 4$, is given by 
\begin{center}
 $(z^{k-2} + iaz^{k/2}) dz^2$ \hspace{.1in} for $a\in \mathbb{R}$.
 \end{center}
Here the critical graph has finite-length edges in addition to the $k$ infinite prongs (see \cite{HubbMas} or  Lemma 1.1 of \cite{AuWan}).  For more examples, see \S2 of \cite{Gup25}.

\section{Planar ends and model maps}

In this section we introduce the space of $k$-planar ends $\P(k)$ and their harmonic collapsing maps to the graph $X_k$ (that has a single vertex $O$ and $k$ infinite-length prongs). Together with an angle of rotation, these will form the space of model maps $\M(k)$ around the puncture $p$, for the collapsing maps of half-plane differentials on the punctured Riemann surface $\Sigma\setminus p$.  In \S3.2 we highlight a topological property of these model maps that shall be useful later. 

\subsection{Definitions.} Throughout, we fix an integer $k\geq 1$.

\begin{defn}[Planar end]\label{pend}  A \textit{$k$-planar end}  $P$ is  conformally a punctured disk with a singular-flat metric, obtained by gluing $k$ Euclidean upper half-planes to each other by an interval-exchange of a finite number of subintervals of their boundaries, where conical singularities of angle $\pi$ in the resulting metric are not allowed. (see Figure 3). 
\end{defn}

 \begin{figure}[h]
  \centering
  \includegraphics[scale=0.6]{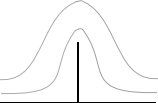}\\
  \caption{A fold is not allowed in the gluing on the boundaries.   }
\end{figure}

The unglued intervals form the boundary of the resulting punctured disk (See Figure 4.).  The resulting surface has infinite area, with $k$ rays incident to the puncture at infinity;  each ray corresponds to a pair of rays on the boundaries of two half-planes that are identified by the gluing.

 (Note that since there are only finitely many intervals, each half-plane boundary has two half-infinite rays, and because the result of the gluing is to be a punctured disk, the Definition~\ref{pend} implicitly requires each half-infinite ray to be glued to another in a manner that induces a cyclic ordering on the half-planes.)
 
 \begin{definition} \label{defn: P(k)}
 We define $\P(k)$ to be the space of planar ends, in the topology induced by the metric: two planar ends are close if there is a bi-Lipschitz map of small distortion between them.
 \end{definition}

\begin{prop}\label{pk} The space $\P(k)$ of $k$-planar ends is homeomorphic to $\mathbb{R}^{k-1} \times \mathbb{R}_+ \cong \mathbb{R}^k$.
\end{prop} 

\medskip
The idea of the proof (given below in detail) is as follows: the gluings of the boundaries of the $k$ half-planes yields a planar metric graph that forms a spine for the punctured disk (and includes its boundary).  We describe the parameter space of such graphs, using a result in \cite{MulPenk}. Briefly, there are different combinatorial possibilities of the metric graph, one obtainable from another by Whitehead moves along the finite-length edges. For each combinatorial type, the lengths of the finite edges parametrize a  simplicial cone,  and these piece together to give a space homeomorphic to $\mathbb{R}^{k-1} \times \mathbb{R}_+$, where the positive real factor determines the overall ``scaling".\\

\begin{figure}
  \centering
  \includegraphics[scale=0.27]{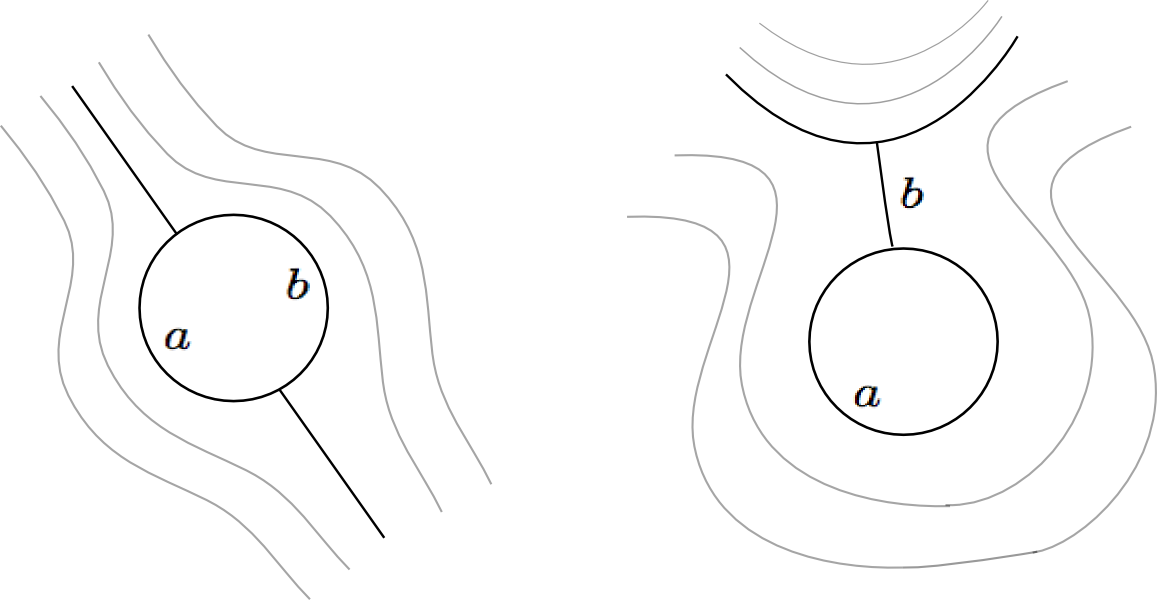}\\
  \caption{A $2$-planar end has two possible patterns of metric spines. The non-negative edge lengths $(a,b)$ for each pattern parametrize cells that fit together to form $\mathbb{R}^2$. (See Propn. \ref{pk}.) }
\end{figure}

Following \cite{MulPenk} we first define:

\begin{defn}\label{mexp} A \textit{metric expansion}  of a graph $G$ at a vertex of degree $d\geq 3$ is a new graph obtained by replacing the vertex by a tree (with each new vertex of degree greater than two) that connects with the rest of the graph.  (See Figure 5 for an example.) 
\end{defn}

The following proposition is culled from Theorem 3.3 of \cite{MulPenk}:

\begin{prop}\label{mex} The space of metric expansions of a $d$-pronged tree $X_d$ is homeomorphic to $\mathbb{R}^{d-3}$.
\end{prop}
\begin{proof}[Sketch of the proof]
A generic metric expansion replaces the vertex with a tree that is dual to a triangulation of a $d$-gon on the plane, which in turn arises as the convex hull of a function on the vertices. Postcomposing such a function with the restriction of an affine map in $\text{Aff}(\mathbb{R}^2,\mathbb{R}) \cong \mathbb{R}^3$ to the vertices does not affect the convex hull. Hence the space of metric expansions is obtained by quotienting out the space of functions on the $d$-vertices by the space of these affine maps, resulting in $\mathbb{R}^{d-3}$.
\end{proof}

\begin{proof}[Proof of Proposition \ref{pk}]
For each $a \in \mathbb{R}_+$, consider the single-vertex graph $\Gamma_a$ consisting of a loop of length $a$ and $k$ (infinite-length) rays from the vertex.  Next consider the space of metric expansions of $\Gamma_a$ for each $a\in \mathbb{R}_+$. Since the vertex has valence $(k+2)$, 
(see Figure 5 for the case $k=3$), by Proposition \ref{mex} the space of metric expansions of $\Gamma_a$ is homeomorphic to $\mathbb{R}^{k-1}$ for each choice of $a$.  The total space $\mathsf{S}$ of such metric expansions (when we vary $a$)  is then $\mathbb{R}^{k-1} \times \mathbb{R}_+$.  

\begin{figure}
  \centering
  \includegraphics[scale=0.4]{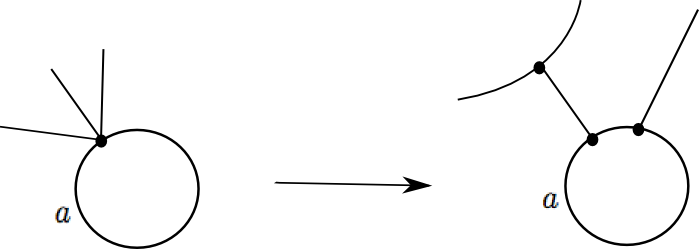}\\
  \caption{A metric expansion at a vertex of degree $5$.   }
\end{figure}

Finally, note that for any such metric graph, one can attach $k$ half-planes to obtain a $k$-planar end, and conversely, the metric spine of a $k$-planar end is a graph in $\mathsf{S}$.  These maps are clearly inverses of each other, and each is continuous: this thus  establishes a homeomorphism between  $\mathsf{S}$ and $\P(k)$. 
\end{proof}

In what follows we shall need:

\begin{defn}\label{collapse}
The \textit{collapsing map} of a planar end  $P$ is the harmonic map to the $k$-pronged tree $X_k$
 \begin{center}
 $c_P:P \to X_k$
 \end{center}
obtained by mapping each half-plane to a prong by the map $z\mapsto \Im(z)$ that collapses the horizontal  foliation to its leaf space. (Throughout the article, $\Im(z)$ and $\Re(z)$ shall denote the imaginary and real parts, respectively of a complex number $z$.)

Note that this map takes the boundary of the planar end, and $k$ arcs to the puncture, to the vertex $O$ of the tree $X_k$. 
\end{defn}

\begin{defn}[Model maps]\label{modmap}
 Let $P \in \P(k) $ be a planar end and $\theta \in [0,2\pi/k)$ be an angle, which we shall henceforth consider as lying on a circle $S^1$.  Consider the conformal (uniformizing) homeomorphism  $\phi_\theta:P\cup \{\infty\}  \to \mathbb{D}$  such that:
\begin{itemize}
\item $\phi_\theta$ takes $\infty$ to $0$, and
\item $\phi_\theta$  takes the $k$ boundary rays to arcs incident to the puncture at (asymptotic) angles $\exp({i\theta}), \exp(i\theta + \frac{i2\pi}{k}), \cdots ,\exp(i\theta + \frac{i2\pi(k-1)}{k})$.  
\end{itemize}

(Note that the first condition determines such a uniformizing map $\phi_\theta$ up to rotation, and the latter determines it uniquely.) 
Then the \textit{model map} for this local data $L = (P,\theta) \in \P(k) \times S^1$ is the harmonic map 
\begin{equation}\label{mmap}
m= c_P \circ \phi_\theta^{-1}:\mathbb{D}^\ast \to X_k
\end{equation}
where $c_P$ is the collapsing map for $P$.
The collection of such model maps with the compact-open topology shall be denoted by $\M(k)$.
\end{defn}

By definition, the assignment $L=(P, \theta) \mapsto m$ above provides a  homeomorphism $\P(k) \times S^1 \cong \M(k)$. Hence by  Proposition \ref{pk} we have the homeomorphism $\M(k) \cong \mathbb{R}^{k} \times S^1$  which is part of the statement of Theorem \ref{main}. As a remark, we can also write the right-hand side $\mathbb{R}^{k} \times S^1$ of this homeomorphism as  $\mathbb{R}^{k-1} \times \mathbb{C}^\ast$ by conflating  $\mathbb{R}_+ \times S^1$ to $\mathbb{C}^\ast$ (together they represent the ``complex scale" of the  differential at the pole).

\subsection*{Two properties.} 

\begin{enumerate}
\item First, we note that a model map is harmonic in the sense of Definition \ref{har}.  In particular, consider an annular subsurface $S$  of the planar end $P$ with a one boundary component $\partial S \cong S^1$ that links the puncture, and the other boundary component equal to $\partial P$. 
Then, the restriction of the collapsing map of $P$  to the subsurface $S$ defines a harmonic map to a (finite) subtree, say $\chi$, that takes $\partial P$ to the vertex $O$ along with some boundary map $\partial S \to \chi$. The uniqueness of harmonic maps to trees (or more generally, non-positively curved targets - see \cite{Mese}) then implies that the collapsing map in fact solves the Dirichlet problem  (and is the energy minimizer) for these induced boundary maps. See Lemma \ref{dirich} for an instance of this.\\
\item The collapsing map from $S$ to $\chi$ as in (1) has the feature that any interior point of a prong of $\chi$ has exactly two pre-images on the boundary component distinct from $\partial P$. We shall call this property \textit{``prong duplicity"}. Note that this property holds not just for a model map, but also for the collapsing map for any half-plane differential on $(\Sigma,p)$, when restricted to such an annulus $S$ around $p$ in the corresponding planar-end. \\
\end{enumerate} 
We shall exploit these properties  in  Lemma \ref{toplem} of the next section.

\subsection{Topology of the collapsing foliation.}

As we saw in the previous section, the model maps collapse along the leaves of a foliation on the punctured disk (or its restriction to an annular subsurface $S$). This foliation has finitely many prong-type singularities, with the property that all these singularities lie on a connected ``critical graph". In subsequent constructions (in Proposition \ref{first} and Lemma \ref{hconn}),  we will need to ensure that this connectedness  holds for the  critical graphs for the Hopf differentials for a sequence of harmonic maps, and for the limiting map. To this end, we shall use the following technical lemma. It shows that the connectedness is forced by the two properties -  harmonicity and the ``prong-duplicity" property  - of the maps (see the end of \S3.1). 

The underlying reason for this is as follows:  the prong-duplicity condition rules out any ``folding" of the map (which would have generated more pre-images)  -  the prong-singularities (see Definition~\ref{defn: measured foliation})  are then forced to map to the vertex as it is the only ``singular" point of the target;  finally, the Maximum Principle implies that the preimage of the vertex is connected (this implicitly uses the feature that the target is a tree).

\begin{lem}[Topological lemma]\label{toplem} Let $A$ be an annulus with boundary components $\partial^{\pm}A$, and $\chi$ be a finite $k$-pronged tree with a single vertex $O$. Let $f:A\to \chi$ be a harmonic map collapsing along the leaves of a foliation $F$ on $A$ that is smooth except for finitely many prong-singularities. Further, assume that  $f$ satisfies the following conditions:
\begin{itemize}
\item \textnormal{(Prong Duplicity)} Each interior point of a prong of $\chi$ has precisely two preimages on $\partial^{+}A$, and 
\vspace{.07in}
\item $\partial^{-}A$ maps to the vertex $O$ of $\chi$.
\end{itemize}
Then the preimage $f^{-1}(O)$ is a connected graph that contains all the singularities of $F$.
\end{lem}

\begin{proof}
As we shall see, the harmonicity and the conditions above place restrictions on the possible behaviors of the level sets, which are leaves of $F$. 

First, note that as  a consequence of the fact that $f$ is harmonic, by the Maximum Principle, no leaf of $F$ can bound a simply connected region in $A$. 

The argument for the connectedness of $f^{-1}(O)$ also uses the Maximum Principle: 

Consider the level set $f^{-1}(O)$, and assume it is not connected.  Let $\Gamma_1$ and $\Gamma_2$ be two disjoint connected components of $f^{-1}(O)$, with $\Gamma_1$ containing $\partial^-A$. We may choose the ``innermost" such pair,  so that, together with arc(s) on $\partial^+ A$, the graphs $\Gamma_1$ and $\Gamma_2$ bound a region of $A$ without any preimage of the vertex in its interior.

\begin{figure}[h]
  \centering
  \includegraphics[scale=0.31]{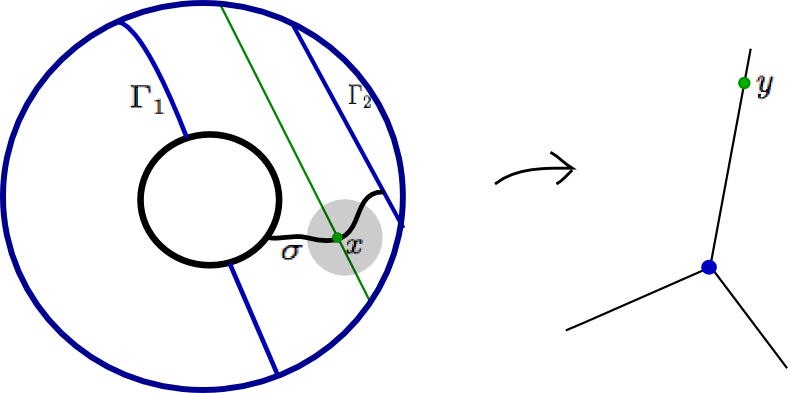}\\
  \caption{The (hypothetical) case when $f^{-1}(O)$ is disconnected. Here the inner boundary of the annulus $A$ (shown on the left)  is $\partial^+A$ and the outer boundary is $\partial^-A$.}
\end{figure}

Consider a path $\sigma$ connecting $\Gamma_1$ and $\Gamma_2$. Since the image of $\sigma$ does not cross the vertex, the loop $f(\sigma) \subset \chi$ remains on one prong. 
Let $y$ be the point on $f(\sigma)$ at maximum distance from the vertex, and let $x\in \sigma$ map to $y$. (See Figure 6.). Consider a small disk $D$ centered at $x$ contained in the interior of the region of $A$ bounded by $\Gamma_1$ and $\Gamma_2$. The restriction of the harmonic map $f$ to $U$ yields a subharmonic distance function $d(f(z), O)$ on $U$ with a maximum achieved in its interior. By the Maximum Principle, this forces $f$ to be constant on $U$, which is absurd, since its level sets are leaves of $F$. \\

Moreover, for an interior point of a prong, we can also show:\\

\textit{Claim. If $y\in X_k$ is an interior point of a prong, then the preimage $f^{-1}(y)$ is a smooth arc without singularities and has exactly two endpoints on $\partial^+A$.}\\
\textit{Proof of claim.} Otherwise, there will be a singularity, and hence some branching at some point $x^\prime \in f^{-1}(y)$. Consider two of the branches of the preimage: if they meet each other at an interior point, the resulting closed curve will either enclose a simply connected region, which is not allowed by the Maximum Principle, or separate the boundary components $\partial^\pm A$ which contradicts the fact that the preimage of the vertex $O$ is connected and intersects both. The same argument rules out the two branches reaching the same point on $\partial^+A$. Hence they reach two distinct points of $\partial^+A$.  Then, the  third branch will either
\begin{enumerate}
\item reach $\partial^+A$, which will violate prong-duplicity. (See Figure 7.) Note that it must reach a point distinct from the other two branches as otherwise the two branches from $x^\prime$ with the same endpoint on $\partial^+A$ will create a bigon, also disallowed by the Maximum Principle.  
\item close up, violating the Maximum Principle, as mentioned before,
\item reach $\partial^{-}A$, which will violate the fact that  $\partial^{-}A$ maps to $O$.
\end{enumerate}

\begin{figure}[h]
  \centering
  \includegraphics[scale=0.3]{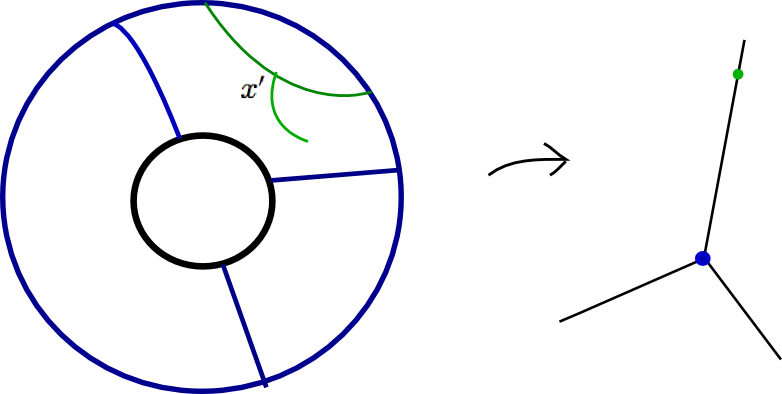}\\
  \caption{A hypothetical singularity at $x^\prime$ in $A$ (left) is mapped to an interior of a prong in $X_k$ (right).}. 
\end{figure}

Hence in each possibility we get a contradiction. $\qed$\\

The claim just proved then implies that all the singularities must lie on the preimage of the vertex $f^{-1}(O)$, which we also showed was connected.
\end{proof}

\subsection{Bounded distance $\implies$ identical}
Lastly, we  observe:

\begin{lem}\label{unbdd} A distinct pair of model maps $m_1,m_2 \in \M(k)$ are an unbounded distance from each other, namely the distance function $d:\mathbb{D}^\ast \to \mathbb{R}$ defined by $d(z) = d_{X_k}(m_1(z),m_2(z))$ is unbounded.
\end{lem}
\begin{proof}
Assume that $d$ is bounded. Since the distance function 
$d_{X_k}(\cdot,\cdot): X_k \times X_k \to \R$ is convex, the pullback $d$  by the harmonic map $(m_1,m_2): \mathbb{D}^\ast \to X_k \times X_k $ is subharmonic. Moreover, since any model map takes $\partial \mathbb{D}^\ast$ to the vertex of $X_k$, we have that $d\vert_{\partial \mathbb{D}^\ast} \equiv 0$.  As $\mathbb{D}^\ast$ is parabolic (in the potential-theoretic sense), a bounded non-negative subharmonic function on $\mathbb{D}^\ast$ that vanishes on the boundary is identically zero (see IV. \S1 6C,7E and \S2 9 of \cite{AhlSar} or  Theorem X.7 and X.17 of \cite{Tsuji}). Thus $d$ vanishes identically, contradicting the fact that $m_1$ and $m_2$ are distinct.
\end{proof}

\textit{Remark.} This last lemma explains why the ``geometric-analytic" space in Theorem \ref{main} involves the asymptotic behavior of harmonic maps upto ``bounded distance". In particular, for any fixed model map, there is a \textit{unique} harmonic map on the surface asymptotic to it in this sense. This shall be used in \S5.2.

\section{Symmetric annuli and estimates for least-energy maps}

Recall that $X_k$ is a tree with a single vertex $O$ and $k$ prongs of infinite-length. In what follows,  a \textit{finite $k$-pronged subtree of $X_k$} (usually denoted by $\chi$) will mean a subset of $X_k$ with the single vertex $O$ and $k$ finite prongs obtained by truncating each prong of $X_k$.\\

 Consider the following two boundary-value problems. Here, and subsequently in the paper, we shall implicitly assume that maps considered are continuous with weak derivatives locally in $L^2$ (see Definition \ref{har}). \\

\textbf{Problems.} Find the energy-minimizing map from a conformal (round) annulus $A$ to a finite $k$-pronged sub-tree of $X_k$ with a prescribed map on the boundary component $\partial^+A$, and with either:
\begin{itemize}
\item \ a prescribed map on the boundary component $\partial^- A$ ( a \textit{Dirichlet boundary problem}); when this is equal to the constant map to the vertex $O$  we call it a \textit{Dirichlet-$O$ boundary problem}, or 
\item no requirement on the boundary component $\partial^-A$  (a \textit{ partially free boundary problem}). 
\end{itemize}
\medskip

\textit{Remark.} We note that there exists a solution to the partially free boundary problem for each annulus $A$:  consider solutions of Dirichlet problems for different boundary problems and then restrict to a sequence of those whose energy tends to an infimum. Since there is an  overall energy bound, the Courant-Lebesgue Lemma provides for the equicontinuity of that sequence.  Because the  boundary componeny $\partial^-A$  has fixed image,  Lemma \ref{alem1} applies, and shows  that  there is a convergent subsequence.\\ 
 
In this section we show that under certain symmetry assumptions, the energies of the solutions of these two problems differ by a bounded amount (Proposition~\ref{energy}). This comparison shall be crucial in  the final section of \S5 in proving a uniform energy bound. \\

\textbf{Outline of the section.} The results of this section shall be used later (\S5.4) to extract a convergent subsequence from a sequence of harmonic maps to $X_k$ defined on a  compact exhaustion of our  punctured Riemann surface $\Sigma\setminus p$.  Such a convergence is not obvious, as the images of such a sequence of maps are not uniformly bounded - they lie in the  graph $X_k$, and, as the sequence progresses, the images cover arbitrarily large portions of the infinite-length prongs.  In this section we shall focus on an exhaustion of a planar end by symmetric annuli (defined in \S\ref{sec:symmetric exhaustion}) and eventually derive uniform bounds on compact sets, for the energy and diameter of the partially free boundary solution for any such annulus in our exhaustion. The main comparison results are stated in \S\ref{sec:Dirichlet Partially Free}, where we compare the solutions to the Dirichlet and the partially free boundary value problems on the annuli (Propositions \ref{main-prop} and \ref{energy}).

The first key observation (\S\ref{sec:doubling}) for the proof is that the partially free boundary problem on an annulus reduces to a Dirichlet problem on the doubled annulus (where the doubling is across the ``free" boundary). Then in \S\ref{sec:decay}, we use a standard decay for  harmonic functions defined on a cylinder (here the decay is towards the central circle of $A$) to uniformly bound -- with a quantitative decay estimate -- the image of the ``central" curve of the cylinder.  Finally, in \S\ref{sec:Proof prop 4.7}, we apply this estimate in our setting to show that this decay compensates enough for the growth of the image of the maps resulting from the growth of the compact subsurfaces: this last balance of inequalities uses the additional symmetry we have assumed.

\subsection{Exhaustion by symmetric annuli.} \label{sec:symmetric exhaustion}
We begin with some definitions.

\begin{defn}[Rectangular annulus]\label{rann} A \textit{rectangular annulus} $A$ in a planar end $P$ is an annular subsurface with a core curve linking the puncture, whose boundary consists of alternating horizontal and vertical edges on each half-plane. A \textit{truncation} of a planar end $P$ is a rectangular annulus with one boundary component precisely equal to the  boundary of $P$.
\end{defn}

\begin{figure}
  \centering
  \includegraphics[scale=0.4]{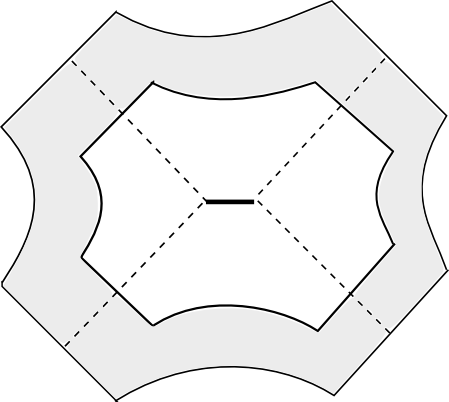}\\
  \caption{A symmetric rectangular annulus (shaded) in a $4$-planar end $P$ has at least a $2$-fold symmetry. Here the boundary of the half-planes for the planar end are shown dotted, and the two sides of a horizontal slit in the middle (in bold) form its boundary $\partial P$. }
\end{figure}

Let $A$ be a rectangular annulus in a planar end $P$. Then the restriction of the collapsing map  $c_P$ (as in Definition \ref{collapse}) to $A$ is a harmonic map to its image in  $X_k$ (see Property (1) following Definition~\ref{modmap}). In particular, if $A$ is a  truncated planar end, then $c_P$ is a harmonic map to the finite $k$-pronged subtree $\chi$ with prong lengths given by the lengths of the horizontal edges in the outer boundary. In this case, the collapsing map $c_P$ takes the boundary component $\partial^-A$  to the vertex $O$, and so solves the ``Dirichlet-$O$ boundary" problem. We summarize the discussion in the following lemma. 

\begin{lem}\label{dirich}
The collapsing map $c_P$ restricted to a truncated planar end $A$ in a planar end $P$ is the least energy map amongst those maps to $X_k$ with same values on the (outer) boundary component $\partial^+A$, and which map the (inner) boundary $\partial^-A$ to $O$. 
\end{lem}


\textit{Remark.} The energy of the collapsing map  is equal to one-half the total area of the rectangles constituting $A$.

\begin{defn}\label{symmann} A \textit{symmetric} rectangular annulus in a $k$-planar end is a rectangular annulus that 
uniformizes to a round annulus such that the restriction of the collapsing map to the boundary components have a $k$-fold rotational symmetry when $k$ is odd, and $k/2$-fold rotational symmetry when $k$ is even. (See Figure 8.) In particular, if the subtree $\chi \subset X_k$ is the image of $c_P$, then $\chi$ has that same order of symmetry.
\end{defn}

Finally, we shall  need the following notion:

\begin{defn}[Symmetric exhaustion]\label{symex} A \textit{symmetric exhaustion} of a planar end $P$ is a sequence of nested symmetric rectangular annuli whose union includes a neighborhood of the puncture at infinity. 
\end{defn}

\textit{Example.} We refer to the example in \S1 following Theorem \ref{main}. On the  punctured disk $\mathbb{C} \setminus \{0\}$ consider the meromorphic quadratic differential $z^{k-2}dz^2$  where $k\geq 2$. The induced metric is that of a $k$-planar end, and clearly has a $k$-fold symmetry (multiplying $z$ by a $k$-th root of unity leaves the differential invariant). In particular, we can choose a sequence of closed curves with alternating vertical and horizontal segments, having this symmetry, that exhaust the end and bound a sequence of symmetric annuli. It is well-known (see \cite{Streb}) that any odd higher-order pole has a coordinate chart where the differential is of this form, and hence, by the above description, has a symmetric exhaustion.

 In what follows we give an independent synthetic-geometric proof of the existence of a symmetric exhaustion, that works in general.

\begin{lem}Any planar end $P$ has a symmetric exhaustion.
\end{lem}

\begin{figure}
  \centering
  \includegraphics[scale=0.4]{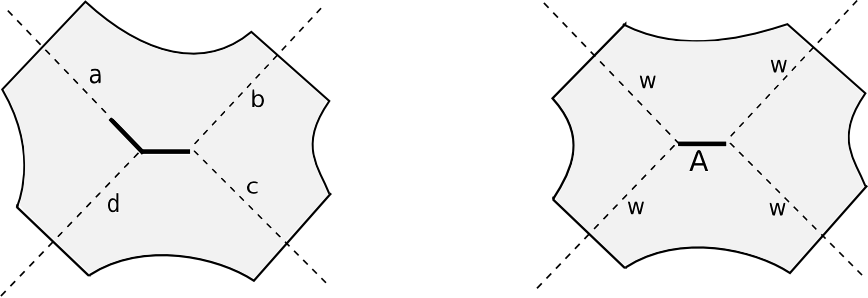}\\
  \caption{As in Figure 8, a  planar-end can be obtained by introducing a slit (shown in bold) on a conformal copy of $\mathbb{C}$ obtained by gluing half-planes. Given a  $4$-planar end (left) one can always choose a symmetric $4$-planar end (right) and truncations of each, such that their complements are isometric.}
\end{figure}

\begin{proof}
We shall work with the metric spines of the planar ends - the latter is easily recovered from the former by attaching half-planes. 
A  symmetric exhaustion is obvious for a ``symmetric $k$-planar end", that is, one with a critical spine having the required symmetry (as in the preceding example). In what follows, we reduce the general case to this.  

 A ``truncation" of a metric spine $S$ of $P$ shall mean the subgraph obtained by cutting off an end of each infinite-length ray. An \textit{$H$-thickening} of such a truncated metric spine (for some $H>0$)  is obtained by  taking the points in $P$ whose vertical distance from the truncated spine is not greater than $H$.  This is then a subsurface of $P$ comprising rectangles of height $H$ on each half-plane. 
 
 The main observation is that, given an arbitrary planar end $P$, we can choose a  symmetric $k$-planar end $P_{sym}$, and truncations of their metric spines, such that the complements of their $H$-thickenings are  \textit{isometric}.
 
To see this, we apply the proof of  Lemma 3.15 in \cite{HubbMas}: namely, by elementary linear algebra (Lemma 3.13 of \cite{HubbMas}) we choose the edge-lengths of the symmetric $k$-planar end $P_{sym}$, and distances along the rays for truncating both $P$ and $P_{sym}$, such that  for any thickening the edge-lengths of the resulting rectangles would match. (See Figure 9.) 

Pulling back a symmetric exhaustion of the latter by this isometry then produces the desired exhaustion on the planar end $P$ we started with.
\end{proof}

\subsection{Comparing solutions of Dirichlet and partially free boundary value problems.} \label{sec:Dirichlet Partially Free}

The setting for this subsection is the following. Suppose $\{A_i\}_{i\geq 1}$ is a symmetric exhaustion of a planar end $P$. Let $h_i:A_i \to X_k$ be the harmonic map that solves the partially free boundary problem with the map on the outer boundary being the restriction of the collapsing map $c_P$. Also, let $c_i:A_i \to X_k$ be the restriction of $c_P$ to the annuli; we have already seen that they solve the corresponding Dirichlet boundary problem on $A_i$ (see Lemma \ref{dirich} for the case when the exhaustion is by truncations of $P$).  Our aim in this section is to compare these two harmonic maps; in particular,
our main goal in this section \S4 is to prove

\begin{prop}\label{energy} Let $\{A_i\}_{i\geq 1}$ be a symmetric exhaustion of a planar end $P$. Let $c_i:A_i\to X_k$ be the restrictions of the collapsing map $c_P$, and $h_i:A_i \to X_k$ be the solutions of the  partially free boundary problem as above. Then their energies satisfy:
\begin{equation}\label{ebound1}
\mathcal{E}(h_i) \leq \mathcal{E}(c_i)  \leq \mathcal{E}(h_i) + K
\end{equation}
where $K$ is independent of $i$. 
\end{prop}

This proposition will be used in the final section of \S5 to prove an important part of the main theorem.\\

The key ingredient in the proof of Proposition~\ref{energy} is 

\begin{prop}\label{main-prop} Suppose $\{A_i\}_{i\geq 1}$ is a symmetric exhaustion of a planar end $P$. Let $h_i:A_i \to X_k$ be the harmonic map that solves the partially free boundary problem with the map on the outer boundary being the restriction of the collapsing map $c_P$. Then the distance from the  image  under $h_i$ of the common (inner) boundary component to the vertex $O$ is uniformly bounded (independent of $i$).
\end{prop}

The proof of this proposition will occupy subsections \S4.3-4.5. The $k$-fold symmetry of the symmetric annuli  plays a crucial role in \S4.4, to guarantee the uniform bound for any  ($k$-) planar end: essentially,  the argument then reduces to that of a $2$-planar end, with a general case being a finite cover.  \\

Given Proposition~\ref{main-prop}, the proof of Proposition~\ref{energy} is straightforward, so we give it now.

{\it Proof of Proposition~\ref{energy}:}

Consider a fixed  symmetric collar neighborhood $N$ of the common (inner) boundary $I$ of $A_i$. (We shall assume, for convenience, that $I$ coincides with $\partial P$.) Then Proposition \ref{main-prop} applied to the symmetric exhaustion starting from the outer boundary of $ N$ implies that the image of  this collar neighborhood by the sequence $h_i$  has a uniform diameter bound, say by $C$. 

As before, let $c_i:A_i\to X_k$ be the restriction of the collapsing map $c_P$.

We construct a candidate map $g:A_i\to \chi_i$  for the ``Dirichlet-$O$ boundary problem" by
\begin{itemize}
\item setting $g$ to equal to $h_i$ away from the fixed collar $N$.
\item interpolating the map on $N$ so that the inner boundary $I$ is mapped by $g$ to the vertex $O$.  (In case the boundary $I$ does not coincide with the boundary $\partial P$ of the planar end,  the image will be a (fixed) finite $k$-pronged subtree of $X_k$, and we instead interpolate to the corresponding map to it.)  
\end{itemize}

For such a map $g$, we see that $g$ agrees with $c_i$ on the boundary $\partial A_i$ and is a candidate for the Dirichlet boundary problem whose energy is minimized by $c_i$. Hence we may conclude that 
\begin{equation*}
\mathcal{E}(c_i) \leq \mathcal{E}(g).
\end{equation*}
On the other hand, since the diameter of the collar is bounded (with bound independent of the index $i$), the interpolation of the map $g$ over the neighborhood $N$ can be done with a bounded cost of energy, independent of $i$. Hence
\begin{equation*}
\mathcal{E}(g) \leq \mathcal{E}(h_i) + K
\end{equation*}
which gives the right-hand inequality of \eqref{ebound1} once we apply the previous inequality.

To conclude, we note that the left-hand inequality in \eqref{ebound1} is immediate: since the collapsing map $c_i$ is a candidate for the partially free boundary minimizing problem solved by $h_i$, we find 
$\mathcal{E}(h_i) \leq \mathcal{E}(c_i)$. \qed

\subsection{Doubling trick} \label{sec:doubling}

Our goal then is to prove Proposition~\ref{main-prop}. We first show:

\begin{prop}\label{doub} Let $A$ be a conformal annulus and let $\chi$ be a finite $k$-pronged subtree of $X_k$.  Fix a continuous map $\phi:\partial^+A\to \chi$ on one boundary component that takes on each value only finitely often, and consider the solution $h:A\to \chi$  to the  partially free boundary problem that requires $h$ to agree with $\phi$ on that boundary component $\partial^+A$. Then this map $h$ extends by symmetry to a solution $\hat{h}$ of the symmetric Dirichlet-problem on a doubled annulus $\hat{A} = A^+ \sqcup_{\partial_{-}A} A^-$ where one requires a candidate $\phi$ to be the map on both boundary components of $\hat{A}$. In particular, we have $h = \hat{h}\vert_{A^+}$. 
\end{prop}

Note that the solution to the  partially free boundary value problem exists (see the beginning of \S4). 

\begin{proof}[Warmup to the proof of Proposition \ref{doub}] We begin by assuming that the solution $h$ to the partially free boundary problem described above has image $h(\partial^-A)$ of the boundary  component $\partial^-A$ disjoint from the vertex of $\chi$.

By this assumption, near the boundary $\partial^-A$, we have that $h$ is a harmonic map to a smooth (i.e. non-singular) target isometric to a segment. 

First, we show that for the solution of the partially free-boundary problem,  the normal derivative at the (free)boundary component $\partial^-A$ vanishes. We include the elementary computation below for the sake of completeness. \\

Consider a family of maps $u_t: A \to \R$ defined for $t \in (-\epsilon, \epsilon)$.  A map $u_0=h$ in this family is critical for energy if
\begin{align*}
0 &=\frac{d}{dt}\Bigr|_{t=0} E(u_t) \\
&= \frac{d}{dt}\Bigr|_{t=0} \int_M |\nabla u_t|^2 dvol_A \\
&= 2\int_M \nabla \dot{u} \cdot \nabla u_0 dvol_A \\
&= -2\int_M \dot{u} \Delta u_0  + \int_{\partial M} \dot{u} \frac{\partial}{\partial \nu} u_0 dvol_A
\end{align*}

Thus, since $\dot{u} = \frac{d}{dt}\Bigr|_{t=0} u_t$ is arbitrary, we see that 
necessary conditions for a solution $u_0$ to the partially free boundary value problem are that 
\begin{align*}
\Delta u_0 &=0 
\end{align*}
\begin{align}\label{normal}
\frac{\partial}{\partial \nu} u_0 &=0.
\end{align}

Thus, the level curves for our map $h=u_0$ to the graph meet the boundary arc orthogonally.

We then show that the partially free boundary solution $h$ is ``half" of a Dirichlet problem on a doubled annulus.  We follow an approach developed by A. Huang in his Rice University thesis \cite{Huang}:
Let $\hat{A}$ be the annulus obtained by doubling the annulus $A$ across its boundary component $\partial^-A$.  That is, if we denote, as usual, the boundary components $\partial A = \partial^+ A \sqcup \partial^- A$, then we set  $\hat{A}$ to be the identification space of two copies of $A$, where the two copies of $\partial^- A$ are identified.  We write this symbolically as $\hat{A}= A \sqcup_{\partial^- A} \bar{A}$, where $\bar{A}$ refers to $A$ equipped with its opposite orientation.

Let $\hat{h}: \hat{A} \to X_{k}$ denote the map defined on $\hat{A}$ that restricts to $h$ on the inclusion $A \subset \hat{A}$ and, in the natural reflected coordinates, on the inclusion $\bar{A} \subset \hat{A}$. By the continuity of $h$ on $A$ and its closure, it is immediate that $\hat{h}$ is continuous on $\hat{A}$.  The vanishing of the normal derivative at the boundary (\ref{normal})  implies that the gradient $\nabla h|_{\partial^- A}$ is parallel to $\partial^- A$.  As that gradient is continuous on $A$ up to the boundary (see e.g. \cite{Evans}, Theorem~6.3.6), we see that $\hat{h}$ has a continuously defined gradient on the interior of the doubled annulus $\hat{A}$. 

Next, note that because $\hat{h}$ is $C^1$ on $\hat{A}$, we have that $\hat{h}_i$ is weakly harmonic on $\hat{A}$.  In particular, we can invoke classical regularity theory to conclude that $\hat{h}$ is then smooth and harmonic on $\hat{A}$.  Thus, since $X_k$ is an NPC space, the map $\hat{h}$ is the unique solution to the Dirichlet harmonic mapping problem of taking $\hat{A}$ to $X_{k}$ with boundary values $h|_{\partial^+ A}$.\qed\\

Next, we adapt this argument to the general case when the image of the boundary $h(\partial^-A)$ might possibly contain the vertex $O$ of the tree $\chi$.  \\

To accomplish the extension to the singular target case, we first analyze the behavior of the level set $h^{-1}(O)$ of the vertex $O$ within the annulus
$A$, particularly with respect to its interaction with the `free' boundary $\partial^- A$. 

\begin{lem}
Under the hypotheses above, any connected component of the level set $h^{-1}(O)$ of the vertex $O$ within the annulus
$A$ meets the free boundary $\partial^- A$ in at most a single point.\end{lem}

\begin{proof}
We begin by noting that the proof of the Courant-Lebesgue lemma, based on an energy estimate for $h$ on an annulus (see, for example, Lemma 3.2 in \cite{Wolf2}) extends to hold for half-annuli, centered at boundary points of $\partial^- A$. Applying that argument yields  a uniform estimate on the modulus of continuity of the map $h$ on the closure of $A$ only in terms of the total energy of $h$. Thus there is a well-defined continuous extension of the map $h$ to $\partial^- A$. We now study this extension, which we continue to denote by $h$.

First note that there cannot exist an arc $\Gamma \subset A \cap h^{-1}(O)$ in the level set for $O$ in $A$ for which $\Gamma$ meets $\partial^-A$ in both endpoints of $\partial \Gamma$.  If not, then since $A$ is an annulus, 
some component of $A \setminus \Gamma$ is bounded by arcs from $\partial^-A$ and $\Gamma$.  But as $\partial^-A$ is a free boundary, we could then redefine $h$ to map only to the vertex $O$ on that component, lowering the energy.  This then contradicts the assumption that $h$ is an energy minimizer.

Focusing further on the possibilities for the level set $h^{-1}(O)$, we note that by the assumption on the boundary values of $h$ on $\partial^+A$  being achieved only a finitely many times,  the level set $h^{-1}(O)$  can meet $\partial^+A$ in only a finite number of points (in fact the number of them  is also fixed and equal to $k$ in subsequent applications, since the boundary map would be a restriction of the collapsing map for a $k$-planar end).

Therefore, with these restrictions on the topology of $h^{-1}(O)$ in $A$ in hand, we see that by the argument in the previous paragraph, each component of $h^{-1}(O)$ then must either be completely within, or have a segment contained in $\partial^-A$,  or  - the only conclusion we wish to permit - connect  finitely many points of $\partial^-A$ with one or more of the finite number of preimages of the vertex on $\partial^+A$.

Consider the first case where a component of $h^{-1}(O)$ is completely contained within $\partial^-A$.  A neighborhood $N$ of a point in such a component then has image $h(N)$ entirely within a single prong, so the harmonic map on that neighborhood agrees with a classical (non-constant) harmonic function to a segment. Thus in a neighborhood of the segment, say on a coordinate neighborhood $\{\Im (z) = y \ge 0\}$, the requirements from equation~\eqref{normal} and that $h(0) = O$ and non-constant require the function $h$ to (i) be expressible locally as $\Im (az^k) + O(|z|^{k+1})$ for some $k \ge 1$ and some constant $a \in \C$, (ii) be real analytic, and (iii) satisfy $\frac{\partial h}{\partial y} = 0$ (where $z = x+iy$).  It is elementary to see that these conditions preclude this segment $h^{-1}(O)$ from being more than a singleton: that  $h^{-1}(O)$ contains a segment defined by $\{y=0, x\in (-\epsilon, \epsilon)\}$ implies that the constant $a$ in condition (i) is real. But then $0=\frac{\partial h}{\partial y} = \Re (az^{k-1}) +O(|z|^{k})$ also on that segment $\{y=0, x\in (-\epsilon, \epsilon)\}$: thus $a = 0$, and so the map $h$ must be constant, contrary to hypothesis.

The same argument rules out  the case when the level set $h^{-1}(O)$ meets the free boundary $\partial^-A$ in a segment, and that segment is connected by an arc of $h^{-1}(O)$ to $\partial^+A$.  But for this situation, we apply the argument of the previous paragraph to a subsegment of $h^{-1}(O)$ on $\partial^-A$ with a neighborhood whose image meets only an open prong, concluding as above that such a segment on $\partial^-A$ is not possible. 

Thus the intersection of such a component of the level set $h^{-1}(O)$ with the free boundary $\partial^-A$ is only a singleton, as needed.
\end{proof}

\emph{Conclusion of the proof of Proposition~\ref{doub}:} It remains to consider the case when the image of the boundary $\partial^-A$ by $h$ contains the vertex $O$.  It is straightforward to adapt, as follows, the argument we gave in the warmup for the smooth case to the singular setting.

 Consider a neighborhood of a point on $h^{-1}(O) \cap \partial^-A$.  Doubling the map on that half-disk across the boundary $\partial^-A$ yields a harmonic map from the punctured disk to the tree (defined everywhere except at the isolated point $h^{-1}(O) \cap \partial^-A$).  That harmonic map is smooth on the punctured disk and of finite energy, and hence has a Hopf differential of bounded $L^1$-norm.  The puncture is then a removable singularity for that holomorphic differential, and hence for the harmonic map. 

The extended map $\hat{h}$ is then harmonic on the doubled annulus, and is the (unique) solution to the corresponding Dirichlet problem, as required. 
Note that  the normal derivative of the map may have a vanishing gradient at the boundary (prior to doubling), this results in a zero of the Hopf differential on the central circle of the doubled annulus.\end{proof}

By the uniqueness of the solution of the  Dirichlet problem on the doubled annulus, we obtain the following immediate corollary of Proposition \ref{doub}:

\begin{cor}\label{uniq} The solution $h:A\to \chi$ of the partially free boundary problem as in  Proposition \ref{doub} is unique.
\end{cor}

\subsection{A decay estimate}\label{sec:decay}

We now use Proposition~\ref{doub} to gain uniform control on the image under the harmonic map $h$ of the free boundary $\partial^{-}A$ (which is the central circle in the doubled annulus $\hat{A}$).

In what follows we will denote by $\Cc(L)$ a cylinder of circumference $1$ and height $L$, parametrized by cylindrical coordinates $(x,\theta)$ where $x \in [0,L]$ and $\theta \in [0,2\pi)$. The \textit{central circle} is the set  $\{(L/2, \theta) \vert \theta \in [0,2\pi) \}$. 

\begin{prop}\label{tech} Let $L>1$ and $h:\Cc(L)\to \mathbb{R}$ be a harmonic function with identical maps $f:S^1 \to \mathbb{R}$ on either boundary that satisfy:
\begin{itemize}
\item the maximum value of $\lvert f \rvert $ is $M$, and
\item the average value of $f$ on each boundary circle is $0$.
\end{itemize}
Then the maximum value of the restriction of $h$ to the central circle is bounded by $O(Me^{-L/2})$, i.e there is a universal constant $K_0$ so that $\lvert h(L/2,\theta) \rvert  \leq   K_0 M e^{-L/2}$, independent of the boundary values $f$ of $h$.
\end{prop}
\begin{proof}
Consider the case when $f(\theta) = Me^{in\theta}$ where $n\geq 1$ and $M$ is a real coefficient. 


We compute that the Laplace equation $\triangle h = 0$ has solution
\begin{equation*}
h(x,\theta)  = \left( \frac{\sinh nx + \sinh n(L-x) }{\sinh nL} \right) Me^{in\theta} 
\end{equation*}
where we have used the boundary conditions $h(0,\cdot) = h(L, \cdot) = f$. 

Thus, at $x=L/2$ we then obtain
\begin{equation}\label{term}
\lvert h(L/2,\theta) \rvert  \leq K \cdot Me^{-\lvert n \rvert L/2}
\end{equation}
for some (universal) constant $K$.

In general, we have the Fourier expansion
\begin{equation*}
f(\theta) = \sum\limits_{n\neq 0} M_ne^{in\theta}
\end{equation*}
where note that there is no constant term because the mean of the boundary map $f$ vanishes.
The coefficients of $f$ satisfy
\begin{equation}\label{fbd}
\sum\limits_{n\neq 0} \lvert M_n\rvert^2 = \lVert f \rVert_2 \leq M^2
\end{equation}
and the general solution is:
\begin{equation*}
h(x,\theta)  =  \sum\limits_{n\neq 0} \left( \frac{\sinh nx + \sinh n(L-x) }{\sinh nL} \right) M_ne^{in\theta} 
\end{equation*}

From (\ref{term}) we find:

\begin{equation}\label{bd1}
\lvert h(L/2,\theta) \rvert  \leq   \sum\limits_{n\neq 0} K\cdot M_ne^{-\lvert n\rvert L/2}
\end{equation}

\medskip 
Note that the geometric series
\begin{equation}\label{geom}
  \sum\limits_{n=1}^\infty e^{-nL} =  \left( \frac{e^{-L}}{1 - e^{-L}} \right)  \leq  (K^\prime)^2e^{-L}
\end{equation}
for the constant $K^\prime = (1-e^{-1})^{-1/2} \approx 1.26$ (once we assume that $L>1$).

By the Cauchy-Schwarz inequality on (\ref{bd1}) and using (\ref{fbd}) and (\ref{geom}) , we then get:
\begin{equation*}
\lvert h(L/2,\theta) \rvert  \leq   K\cdot M \cdot K^\prime e^{-L/2}
\end{equation*}
which is the required bound.
\end{proof}

\subsection{Completing the proof}\label{sec:Proof prop 4.7}

To finish the proof of Proposition \ref{main-prop}, we apply the results of \S4.3 and \S4.4 to the case of harmonic maps from annuli in a symmetric exhaustion of a $k$-planar end to the corresponding finite $k$-pronged subtrees of $X_k$.\\

In what follows, let  $A$ be a symmetric rectangular annulus in a symmetric exhaustion  $\{A_i\}_{i\geq 1}$ in a $k$-planar end $P$. (Recall that these were defined in \S4.1.)  The restriction of the collapsing map for $P$ to the ``outer" boundary $\partial^+A$  forms the boundary condition for the partially free boundary problem. \\

Consider first the case when the number $k$ of prongs is even. In this case, by the $k/2$-fold symmetry of the domain (see Definition \ref{symmann}), the solution $h$  to the partially free boundary problem is the $k/2$-cover of the solution $\bar{h_2}$  of the corresponding problem of a quotient annulus $\bar{A}$  to a $2$-pronged tree $X_2$. This is because, by the uniqueness of the solution to the partially free boundary problem (Corollary \ref{uniq}) on $A$, the solution acquires the same symmetries of the problem,    and then descends to the quotient annulus $\bar{A}$. \\

On  $X_2$, one can define a signed distance function from the vertex, which is linear and hence pulls back, via the harmonic map $\bar{h_2}:\bar{A} \to X_2$ to a harmonic function $d(z) = \pm d_{X_k}(\bar{h_2}(z), O)$ on the annulus ${A}$. 
Then, as we shall quantify below, by Proposition \ref{tech} of the previous section,  this function will have an absolute-value bound on the  boundary component $\partial^-A$. This bound is then acquired by the (usual) distance function of the lifted map on the $k/2$-fold cover.
  \\

To determine the bound $M$ of the distance function for $d$ on the boundary $\partial^-A$, we examine the collapsing map of $\bar{A}$ to $X_2$ on its boundary.  This $M$ is  given by the largest ``height" of the rectangles in the symmetric annulus $A$ (which descend to the two rectangles constituting $\bar{A}$). By its symmetry and well-known estimates, the modulus of $\bar{A}$ is  $\ln {M}$ (up to a bounded additive error), and hence the modulus of $A$ is $\frac{2}{k} \ln M$ (up to a bounded additive error). (See Figure 7.) \\

\begin{figure}
  \centering
  \includegraphics[scale=0.5]{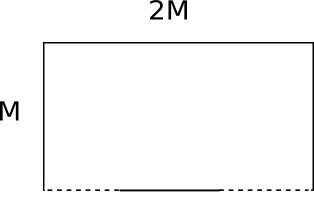}\\
  \caption{For large $M$ the extremal length of the family of arcs between the arcs shown in bold is $\sim 2\ln M + D$ for some real (universal) constant $D$. A truncated $k$-planar end is a $k$-fold branched cover of this.}
\end{figure}

Hence the cylinder $\hat{A}$ in the previous section that we obtained by doubling across the ``free" boundary of $\bar{A}$ has height $L = 2 \ln M  + D$ for some real (universal) constant $D$,  and Proposition \ref{tech} then gives the following  bound on the subharmonic distance function restricted to the middle circle:
\begin{equation*}
\lvert  d (L/2, \cdot ) \rvert \leq K_0 Me^{-L/2} = K_0 Me^{-2\ln M - D}  = K_0e^{-D}M^{-1}
\end{equation*}
which is uniformly bounded for any modulus $M$ large enough (and in fact tends to zero as $i\to \infty$ in the symmetric exhaustion by $\{A_i\}_{i\geq 1}$).\\

Hence the map $\bar{h_2}:\bar{A}\to X_2$ takes the ``free" boundary to a uniformly bounded subset of $X_2$, as $A$ ranges over all the symmetric annuli in the exhaustion. As noted above, the same bound holds for the original map $h:A\to X_{k}$ which is a $k/2$-fold cover of the map $\bar{h_2}:\bar{A}\to X_2$.\\

When $k$ is odd, one needs a small additional step:  the partially free boundary solution $h:A\to X_k$ first  lifts to a double cover $\hat{h}: \hat{A}\to X_{2k}$. The tree $X_{2k}$ now has an even number of prongs and the above argument gives a uniform bound on the image in $X_{2k}$ of the  boundary  component $\partial^-A$ by $\hat{h}$, that is also acquired by the quotient map $h$.\\

This completes the proof of Proposition \ref{main-prop}.

\section{Proof of Theorem \ref{main}}\label{sec:main theorem proof}

In this section, we prove the main theorem.  Recall we need to establish homeomorphisms between the following spaces:

\begin{itemize}

\item \textnormal{(Complex-analytic)} The space of half-plane differentials  $\mathcal{HP}_k(\Sigma,p)$.
\item \textnormal{(Synthetic-geometric)} The space of singular-flat half-plane structures  $\mathcal{P}_k(\Sigma,p)$.
\item \textnormal{(Geometric-analytic)} The space $\mathcal{H}_k(\Sigma,p)$ of harmonic maps asymptotic to \textit{some} model map in $\M(k)$.
\end{itemize}

Here the space of model maps $\M(k)$ (see Definition \ref{modmap}) are defined relative to a uniformization of a neighborhood $U\cong \mathbb{D}$  of the puncture $p$. By ``asymptotic" we mean the distance function between the maps is bounded on $U$. Throughout, the choice of $U$ shall be fixed, and implicit in our discussion.   \\

The spaces above can be given the obvious topologies:  the ``complex-analytic" space $\mathcal{HP}_k(\Sigma,p)$ acquires the topology induced as a subset of the corresponding complex vector space of meromorphic quadratic differentials on $\Sigma$; the ``synthetic-geometric" space $\mathcal{P}_k(\Sigma,p)$ can be given a topology that measures how close the singular-flat metrics are on compact subsets of $\Sigma \setminus p$ (see also Definition \ref{defn: P(k)}); the ``geometric-analytic" space of maps $\mathcal{H}_k(\Sigma,p)$ can be given the compact-open topology.

The maps between these spaces, that we shall subsequently discuss, would all be continuous in these topologies; hence we shall henceforth concern ourselves with showing they define \textit{bijective} correspondences.\\

We shall dispense with the easier correspondence in \S5.1.
To complete the proof of Theorem \ref{main}, we will then be left with showing:

\begin{prop} \label{prop: halfplane to model} The following bijective correspondences hold:
	\vspace{.07in}

\begin{itemize}
	\item \textnormal{(Synthetic-geometric $\rightarrow$ Geometric-analytic)}  The collapsing map for a half-plane structure in  $\mathcal{P}_k(\Sigma,p)$ yields a harmonic map in  $\mathcal{H}_k(\Sigma,p)$.
	\vspace{.07in}
	\item \textnormal{(Geometric-analytic $\rightarrow$ Complex-analytic)} For each model map $m \in \M(k)$, there exists a unique harmonic map $h \in \mathcal{H}_k(\Sigma,p)$ that is asymptotic to $m$, whose Hopf-differential is in  $\mathcal{HP}_k(\Sigma,p)$.
	\end{itemize}
	
\end{prop}

 In \S\ref{sec:HP to M}, we show that the first part of Proposition \ref{prop: halfplane to model}, in \S\ref{sec: model uniqueness} we show the uniqueness statement of the second part, while in \S\ref{sec:model to HP}, we complete the proof of the existence statement of the second part, namely, that of a harmonic map asymptotic to a given model map. This is where we use the energy estimates and the \textit{a priori} bounds established in \S4. 
  
\subsection{Complex-analytic $\leftrightarrow$ Synthetic-geometric.} \label{sec:main-prelims}

Since this part of Theorem \ref{main} is already well-known, our discussion will be brief.

Recall from \S2 that a half-plane differential $q \in \mathcal{HP}_k(\Sigma,p)$ defines a flat singular metric $|q|$ on the surface $\Sigma \setminus p$ that restricts to that of a Euclidean half-plane on each complementary component of the metric spine, hence defining an element of $\mathcal{P}_k(\Sigma,p)$. 

Conversely, given a half-plane structure $\mathsf{S} \in \mathcal{P}_k(\Sigma,p)$,  the quadratic differential $d\zeta^2$ in the natural $\zeta$-coordinate on each Euclidean half-plane $\{\Im\zeta >0 \}$  defines a global holomorphic quadratic differential on the punctured Riemann surface $\Sigma \setminus p$  obtained by an an interval-exchange map on their boundaries. (Note that the maps in an interval exchange are by semi-translations $z\mapsto \pm z+c$.) The image of the boundary-lines  of the half-planes after the identifications forms the critical graph of this resulting differential, which is connected and forms  a metric spine (see the discussion following Definition \ref{critg}). Hence we obtain a half-plane differential  $q \in \mathcal{HP}_k(\Sigma,p)$. 

Clearly, if two such half-plane structures in $\mathcal{P}_k(\Sigma,p)$ are isometric, then the isometry is also a biholomorphism between the underlying Riemann surfaces that also identifies the corresponding  half-plane differentials. 

It is straightforward to check this bijection is also continuous: the singular-flat metrics depend continuously on the half-plane differential.\\

Thus, what remains for the proof of Theorem~\ref{main} is the verification of Proposition~\ref{prop: halfplane to model}.  That proof occupies the next three sections. 

\subsection{Synthetic-geometric $\rightarrow$ Geometric-analytic} \label{sec:HP to M}

Recall that for a half-plane differential, and its corresponding half-plane structure, the collapsing map of the horizontal foliation to its leaf-space defines a harmonic map  $h:\Sigma\setminus p \to X_k$ (see  Definition \ref{collmap}). In this subsection, we show that there is a unique model map asymptotic to that collapsing map, that is:

\begin{prop}\label{first}
The restriction of the harmonic map $h$ to $U$ is bounded distance from a unique model map $m\in \M(k)$ on $U$. Moreover, as $q$ above (and consequently $h$)  varies continuously, so does this model map $m$.
\end{prop}

This proposition provides that the association in the second part of Proposition~\ref{prop: halfplane to model} is well-defined, namely,  there is a well-defined map
\begin{equation}\label{psiu}
\Psi_U : \mathcal{HP}_k(\Sigma,p) \to \M(k)
\end{equation}
that assigns $\Psi_U(h) =m$ in the notation of the above Proposition.
(Here recall that in Theorem \ref{main} we fix a coordinate chart $U\cong \mathbb{D}$ around $p$.) \\

In what follows $q\in \mathcal{HP}_k(\Sigma,p)$ shall be the half-plane differential corresponding to the half-plane structure (we have already established in the previous section that they are in bijective correspondence).\\

Note that for Proposition \ref{first}  we are assuming that $h$ is already a collapsing map of some half-plane differential (and hence of some planar end $P$). The difficulty is that the (arbitrary) choice of the disk $U$ means that its boundary may not coincide with that of the planar end $P$. In particular, the harmonic map $h$ may not take $\partial U$ to the vertex $O$  of $X_k$, as a model map should. The assertion of Proposition~\ref{first} is that nevertheless, the map $h$ is \textit{bounded} distance from such a model map. \\

The strategy of the proof is the following: exhaust the punctured disk $U\setminus p$ with a sequence of annuli $A_1 \subset A_2 \subset \cdots A_n \subset \cdots$, for which one boundary component $\partial^-A$ agrees with $\partial U$. Then, for each such annulus, 
solve a Dirichlet problem to get a harmonic map $m_n:A_n \to X_k$ that maps  $\partial^-A = \partial U$ to $O$, and which restricts to $h$ on the (other) truncated boundary $\partial^+A$. We then show that this sequence of harmonic maps $m_n$ will converge uniformly on compact sets to the required model map $m: U \setminus p \to X_k$.\\

In the above construction,  we need to ensure that the critical graphs for the Hopf differentials for each $m_n$, and the limiting map $m$,  remain connected. Of course, this application is what we had in mind in stating the Topological Lemma~\ref{toplem}. (We recall that this lemma exploits the  ``prong-duplicity" property - see \S3.1- of the collapsing map $h$.)

\begin{proof}[Proof of Proposition \ref{first}]  Recall that $h:\Sigma\setminus p \to X_k$ is the collapsing map for a half-plane differential on $\Sigma \setminus p$, and $U\cong \mathbb{D}$ is a fixed disk centered at the puncture.\\
We shall construct a model map $m:U\to X_k$ as a limit of maps $m_n$ defined on compact annuli exhausting $U\setminus p$.
As usual we identify $U\cong \mathbb{D}$. Consider the annulus $A_n = \mathbb{D} \setminus B(0,1/n)$.  We denote the boundary component $\partial \mathbb{D}$ as $\partial^-A_n$ and the other boundary circle  $\partial B(0, 1/n)$ as $\partial^+A_n$. 

 Let $m_n:A_n \to X_k$ be the energy-minimizing map amongst continuous maps with weak $L^2$-derivatives  that: 
\begin{itemize}
\item restrict to $h$ on $\partial^+A_n$,  and
\item map $\partial \mathbb{D} = \partial^-A_n$ to the vertex $O$.
\end{itemize}

Note that there is such an energy-minimizing map since any minimizing sequence will have a uniform energy bound, and  the image of $\partial^-A_n$ is always the vertex,  so Lemma \ref{alem1} applies. 
Since  any reparametrization of the domain annulus preserves the above properties, the limiting map is stationary (\textit{i.e.} energy minimizing for all reparametrizations of the domain), and by an argument of Schoen  (Lemma 1 of \cite{Schoen}), the Hopf differential is then holomorphic. 
Since $h$ is a collapsing map of a half-plane differential, its restriction to $\partial^+A_n$ satisfies the prong-duplicity condition.  By the Topological Lemma \ref{toplem}, these imply that $m_n$ is a  collapsing map for a foliation whose  singularities which all lie on a connected graph mapping to the vertex $O$ of $X_k$. That is, the Hopf differential of $m_n$ has a connected critical graph.

Next, let the maximum distance of $h(\partial^-A_n)$ from $O$ be denoted by $B\geq 0$. The distance function $d_n:A_n \to \mathbb{R}_{\geq 0}$ defined by
 \begin{center}
  $d_n(z) = d(m_n(z),h(z))$ 
  \end{center}
  is then subharmonic, uniformly bounded by  $B$ on $\partial \mathbb{D} = \partial^-A_n$ and (by construction) equal to zero on $\partial^+A_n$. Hence by the Maximum Principle, the distance function $d_n$ is bounded (by $B$) on $A_n$. Hence we have that all  distance functions are uniformly bounded on any compact set in $\mathbb{D}^\ast$.
  
  In fact, the sequence of harmonic maps $m_n$ are boundedly close to the fixed harmonic map $h$.  Thus, since the map $m_n$ takes the boundary $\partial \mathbb{D}$ to the vertex of $X_k$ for each $n$, then for any compact subset $K \subset \mathbb{D}^\ast$ -- noting that $h(K)$ is fixed independently of $n$ -- we see that there is a uniform bound  on the diameter of its image $m_n(K)$ under $m_n$ (note that $K$ belongs to the domain of $m_n$ for all large $n$).  Hence there is a convergent subsequence  $m_n \to m$ (see Lemma \ref{alem1} in \S2.5) where $m:\mathbb{D}^\ast \to X_k$ is harmonic and, indeed, still at bounded distance from $h$.\\

We need to show that $m$ is a model map in $\M(k)$. We first show:\\

\textit{Claim.  The critical graph of the Hopf differential of $m$ is connected, and is the preimage of the vertex of $X_k$.}\\
\textit{Proof of claim.} We have noted above that the critical graph for each model map $m_n$ is connected. The uniform convergence $m_n\to m$ implies that the Hopf differentials of $m_n$ converge uniformly on compact sets, and so do the horizontal  foliations of those Hopf differentials of $m_n$ (convergence here is in the Hausdorff topology). Let $F$ be the horizontal  foliation for the Hopf differential of $m$. 

The above convergence implies, in particular, that the singularities of the Hopf differentials of $m_n$ converge to singularities of $F$. These account for all the singularities on the punctured disk:  each approximate has precisely $k$ preimages of $O$ on the boundary component $\partial^+A_n$ (that shrink to the puncture), and in the limit the Hopf differential has a pole of order $(k+2)$ at the puncture, so by considering the Euler characteristic,  sum of the orders of the limiting prong-singularities (\textit{i.e.} the total order of zeros) is the same for $m$ as it was for the approximates $m_n$.

 Since for each $n$, the singularities of $m_n$ are mapped to the vertex $O$ of $X_k$, hence so are their limits, which are {\it all} of the singularities of $F$ by the argument above. 

The same argument as in the proof of Lemma \ref{toplem} then completes the proof of the claim:  if there are two components of the critical graph (each mapping to $O$)  then one gets a maximum point of the subharmonic distance function from $O$ in the interior of a region of $A$ bounded by them - a contradiction.$\qed$\\

Thus, the preimage of any interior point of a prong  by $m$ contains no singularities, and is then a bi-infinite leaf with ends in the puncture. 
The complement of the critical graph of the Hopf differential of $m$ is then necessarily a collection half-planes (swept out by the bi-infinite leaves).  Moreover the number of such half-planes is precisely $k$ since the model map $m$ is at most a bounded distance away from the harmonic map $h$ which is itself a collapsing map for some $q\in \mathcal{HP}_k(\Sigma,p)$.  Finally, the map $m$ takes $\partial\mathbb{D}$ to the vertex (since each $m_n$ does). Hence $m \in \M(k)$. \\

The uniqueness follows from Lemma \ref{unbdd} -   two such maps $m_1,m_2$, both a bounded distance from $h$ would be a bounded distance from each other, and hence are identical.  The statement about continuity follows since the singular flat metrics vary continuously when the quadratic differential is varied, and so do their collapsing maps.
 \end{proof}
 
 \textit{Remark.}  This now shows that $\Psi_U: \mathcal{HP}_k(\Sigma,p) \to \M(k)$ defined by $\Psi_U(q) = m$  - see (\ref{psiu}) - is well-defined and continuous.  To prove it is a homeomorphism (and complete the proof of Theorem \ref{main}) it shall suffice to prove that it is a bijection, which the subsequent sections shall accomplish.

\subsection{Uniqueness of the half-plane differential} \label{sec: model uniqueness}
We next prove that the association $\Psi_U$  of half-plane differential to  the model map is injective, namely the uniqueness in the second part of  Proposition~\ref{prop: halfplane to model}. The argument is straightforward:

Suppose there are two half-plane differentials  $q_1,q_2 \in \mathcal{HP}_k(\Sigma,p)$ with harmonic collapsing maps $h_1,h_2:\Sigma\setminus p \to X_k$ that are both asymptotic to the same model map $m$. Then when restricted to $U\cong \mathbb{D}$,  by the triangle inequality the subharmonic distance function $ d(z) = d_{X_k}(h_1(z),h_2(z))$ between the maps is bounded.

 Since the punctured Riemann surface is parabolic in the potential-theoretic sense, such a bounded subharmonic function must be constant, that is $d\equiv c$ for some non-negative real number $c$. (Compare the argument in Lemma~\ref{unbdd}.) 
 
By definition, the preimage of the vertex $O\in X_k$ by $h_1$ and $h_2$ are spines of the punctured surface $\Sigma \setminus p$. In particular, they must intersect (\textit{e.g.} a pair of curves of algebraic intersection number one cannot have disjoint representatives) and hence the distance $d \equiv 0$ on $\Sigma$. Note that this uses the assumption that $\Sigma$ has genus $g\geq 1$.

Hence $h_1 \equiv h_2$, as claimed.  \qed

\subsection{Geometric-analytic $\rightarrow$ Complex-analytic} \label{sec:model to HP} 
In this section we prove the surjectivity of the map $\Psi_U: \mathcal{HP}_k(\Sigma,p)\to \M(k)$  in (\ref{psiu}), namely the  surjectivity  of the second part of Proposition~\ref{prop: halfplane to model}. We shall use here the results of \S4.

Our goal here is to show that for each model map $m \in \M(k)$, there exists a harmonic map $h:\Sigma\setminus p \to X_k$ of bounded distance to $m$, whose Hopf differential $q$ is half-plane, that is, $q \in \mathcal{HP}_k(\Sigma,p)$.  Note that the previous section \S5.3 shows that such a harmonic map $h$ is unique. This would conclude the proof of Proposition~\ref{prop: halfplane to model}.\\

To this end, let  $(\Sigma,p)$, $k\geq 2$, and $U \cong \mathbb{D}$  be as in Theorem \ref{main}. Fix a planar end and an angle $(P,\theta)$ and let $m:U \setminus p  \to X_k$ be a model map in $\M(k)$ determined by this data, as in Definition~\ref{modmap}. 

By Lemma \ref{symex}, we may choose a symmetric exhaustion of $P$
  \begin{center}
   $A_1 \subset A_2 \subset \cdots \subset A_n \subset \cdots$
   \end{center}
   such that the symmetric rectangular annuli are all contained in $U$.

    Let $\Sigma_0 = (\Sigma \setminus U) \cup V$, where $V$ is the interstitial region between $\partial U$ and the inner boundary of $A_1$ (see Figure 11). 
    For $n\geq 1$ define the sequence of Riemann surfaces with boundary:
   \begin{center}
   $\Sigma_n = \Sigma_0 \cup A_n$.
   \end{center}

\medskip

\begin{defn}[$\mathcal{H}_n$] \label{colhn}
For each $n\geq 0$ let $\mathcal{H}_n$  be the set of continuous maps  $h: \Sigma_n\to X_k$  with weak derivatives in $L^2$ such that
\begin{itemize}
\item[(a)] the restriction of any $h$ to $\partial \Sigma_n \subset U$ is the same as the restriction of the  model map ${m}:U\to X_k$, and
\item[(b)] each $h$ is a collapsing map for a  foliation  that is smooth except for finitely many singularities,  all of which lie along a connected spine for $\Sigma_n$ that is mapped to $O \in X_k$.
\end{itemize}
Define ${h_n}:{\Sigma_n}  \to {X_k}$ be an energy-minimizing map amongst all maps in $\mathcal{H}_n$. 
\end{defn}

Note that an energy-minimizing map exists since by Lemma \ref{alem1} any energy-minimizing sequence in $\mathcal{H}_n$ has a convergent subsequence (the images of such a sequence contains the common point $O$).
Such a map has a holomorphic Hopf differential as it is an energy minimizer for all reparametrizations of the domain. We also note:

\begin{lem}\label{hconn} The Hopf differential of $h_n$ has a connected critical graph.
\end{lem}
\begin{proof} By construction, each $h_n$ restricts to the model map $m$ on $\partial \Sigma_n$. 
 Recall that the model map $m$ (and consequently $h_n$) has the ``prong-duplicity property" on the subsurface boundary, namely  that any interior point of a prong of $X_k$ has precisely two preimages on $\partial \Sigma_n$ (see Property (2) after Definition~\ref{modmap}).
In the Topological  Lemma \ref{toplem} we showed that a harmonic map with this property has a Hopf differential with a connected critical graph, in the case that the domain was an annulus with one boundary component mapping to the vertex $O\in X_k$. In the case at hand, the domain is a compact surface with boundary, with a spine that maps to $O$. However,  we can reduce to the case of a punctured disk by making slits along each finite-length edge of the spine. Note that the resulting punctured disk has a boundary that maps to $O$ under $h_n$.  Applying the Topological Lemma then completes the proof.  \end{proof}
 
The main result of this section is:
 
 \begin{prop}\label{limit} After passing to a subsequence, the harmonic maps $h_n$ converge uniformly on compact sets to a harmonic map $h:\Sigma \setminus p \to X_k$,  which is at bounded distance from $m$ on $U$, and whose Hopf differential $q\in \mathcal{HP}_k(\Sigma,p)$.  \end{prop}

The key step in the proof is to show that the energy of the restriction of $h_n$ to $\Sigma_0$ is uniformly bounded:

\begin{lem}[Energy bound]\label{aprio}
There exists a constant $E>0$ such that the energy of the restriction $\mathcal{E}(h_n\vert_{\Sigma_0}) \leq E$ for all $n$.
\end{lem}

 \begin{figure}
  \centering
  \includegraphics[scale=0.37]{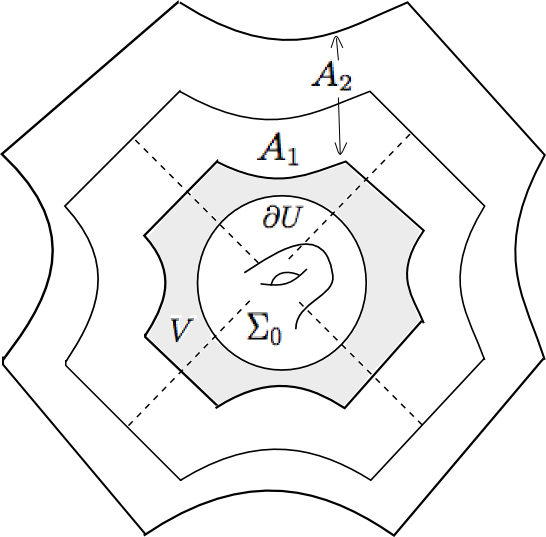}\\
  \caption{Part of the compact exhaustion used in Lemma \ref{aprio}.}
\end{figure}

\begin{proof}
The strategy of proof follows that of \cite{Wolf3}, also used in \cite{JostZuo} - we first describe it in brief:
 Recall that $h_n$ is the energy-minimizing map amongst those in $\mathcal{H}_n$ (Definition \ref{colhn}). The energy $\mathcal{E}(h_n)$ can be expressed as a sum of the energies of the restrictions  of $h_n$ to $\Sigma_0$ and $A_n$.  We  shall bound this  energy in both directions:  for an upper bound we construct a ``candidate" map  $g\in \mathcal{H}_n$, and for  the lower bound we use the estimate on the energy of the partially free boundary problem obtained in  \S4 (Proposition \ref{energy}). The two sides of the inequality shall  then reduce to the required uniform bound.   

We denote by $c_n$ the restriction of the model (collapsing) map $m$ to the symmetric rectangular annulus $A_n$  for $n\geq1$.
Let $e_n$ be the  energy of the  solution of the partially free-boundary problem on $A_n$, that is, the least energy map to the $k$-pronged tree $X_k$  amongst those that restrict to $m$ on the outer boundary $\partial^+A$. 

By Proposition \ref{energy} we have:
\begin{equation}
\mathcal{E}(c_n) \leq e_n + K
\end{equation}
where $K$ is independent of $n$.

Now let  $h_0:\Sigma_0 \to X_k$ the energy minimizing map as defined in Definition \ref{colhn}. \\
For $n\geq 1$ define $g_n:\Sigma_n\to X _k$ to be the map that restricts to $h_0$ on $\Sigma_0$ and to $c_n$ on the annular region $A_n$: note that this map is well-defined as $h_0\Bigr|_{\partial\Sigma_0}=h_0\Bigr|_{\partial^-A_1}=c_n\Bigr|_{\partial^-A_1}$.

Note that $g_n\in \mathcal{H}_n$  and note that the energy $\mathcal{E}(g_n)$ decomposes as
\begin{equation}\label{ebound}
\mathcal{E}(g_n) =  \mathcal{E}(h_0) +  \mathcal{E}(c_n).
\end{equation}

Because $h_n$ is a minimizer for a problem for which $g$ is a candidate, we have
\begin{equation*}
\mathcal{E}(h_n) \leq \mathcal{E}(g_n)
\end{equation*}
and, for analogous reasons, we also have
\begin{equation*}
e_n \leq \mathcal{E}(h_n\vert_{A_n}).
\end{equation*}

Combining these last four displayed inequalities, we obtain:
\begin{equation*}
\mathcal{E}(h_n\vert_{\Sigma_0}) + e_n \leq  \mathcal{E}(h_n\vert_{\Sigma_0})  + \mathcal{E}(h_n\vert_{A_n})  = \mathcal{E}(h_n) \leq     \mathcal{E}(h_0) + e_n + K
\end{equation*}

which implies  
\begin{equation*}
\mathcal{E}(h_n\vert_{\Sigma_0}) \leq     \mathcal{E}(h_0) + K,
\end{equation*}
the required estimate. 
\end{proof}

\textit{Remark.} The same argument applies for the energy of the restriction of the map $h_n$ to any fixed subsurface $\Sigma_{l}$, and hence to any compact set in $\Sigma\setminus p$. 

\begin{cor} After passing to a subsequence, the sequence $\{h_n\}$ converges to a harmonic map $h:\Sigma\setminus p \to X_k$ that is bounded distance from $m$ on the disk $U$.
\end{cor}
\begin{proof}
As in Lemma~\ref{alem1}, the convergence on any fixed compact set is via a standard argument using the Courant-Lebesgue Lemma, using that the elements of the sequence all have the vertex $O$ in their image. A diagonal argument then provides for convergence on full punctured surface $\Sigma \setminus p$. 

Next, on the annulus $A_n \cup V$, by construction we have that the distance function $d_{X_k}(h_n,m)$ vanishes on the ``outer" boundary $\partial^+A_n$. 
By the convergence $h_n\to h$, we have a uniform distance bound (say $d_{X_k}(h_n,m) < D$) on the inner boundary $\partial U$.  Since the distance function is subharmonic on the annulus $A_n$, by the Maximum Principle each map $h_n$ is then a uniformly bounded distance $D$ from the model map $m$, and hence the uniform limit $h$ is also a uniformly bounded distance from $m$. 
\end{proof}

Finally, we observe:

\begin{lem}\label{final}  The critical graph of the Hopf differential of $h$ is connected,  and the Hopf differential  is  a half-plane differential in $\mathcal{HP}_k(\Sigma, p)$. 
\end{lem}

\begin{proof}
Since the convergence of the harmonic maps  $h_n\to h$ is uniform, so is the convergence of the corresponding Hopf differentials. Since these are holomorphic differentials, this convergence is in fact in $C^2$ (and in fact in $C^k$ for any $k$), and hence their critical graphs converge to the critical graph in the limit. By construction, each critical graph in the sequence is connected, and  maps to the vertex $O$.  We briefly recount the argument, identical to that of the ``Claim" in the proof of Propn. \ref{first}, that  implies that  the limiting critical graph has the same property:  

First, the singularities of the approximates limit to singularities of the limiting foliation, with the order of the resulting singularity being (at least) the orders of the approximates. The sum of the orders of these prong-singularities  is determined by the Euler characteristic of the punctured surface; the orders and hence the sum remains the same for the limiting foliation, and hence accounts for all the singularities. There are no other singularities by this Euler-characteristic count, and hence the entire critical graph maps to $O$. Second,  if the critical graph has more than one component, the subharmonic distance function (from $O$) would  have a maximum in the interior of a region enclosed by them, violating the Maximum Principle.  

Thus, the complementary regions of these spines are then necessarily half-planes (see \S2.2) and hence the Hopf differential is a half-plane differential. Moreover since the harmonic map $h$ is bounded distance from the model map $m$ corresponding to a $k$-planar end on their common domain near the puncture $p$, the Hopf differential of the harmonic map $h$ has a pole of order $k+2$ at $p$, and hence is in $\mathcal{HP}_k(\Sigma, p)$.
\end{proof}

These two preceding lemmas complete the proof of Proposition \ref{limit}, and hence the proof of Proposition \ref{prop: halfplane to model}, and hence the proof of Theorem~\ref{main}. \qed 

\section{Proof of Corollary \ref{thm3}}

\begin{defn}[Spaces of ribbon graphs]\label{msg} 
  Let  $\mathcal{MS}_{g,1}^0$ denote the space of marked metric graphs that form an embedded spine of a  punctured surface $(\Sigma, p)$ of genus $g\geq 1$. These are also called ribbon-graphs or  fat-graphs (see, for example, \cite{PennDec}). For $k\geq 1$, let $\mathcal{MS}_{g,1}^k$ denote the space of such marked metric spines that, in addition, have precisely $k$ edges of infinite length that are incident to the puncture at $p$.
\end{defn}

\textit{Remark.} Graphs in $\mathcal{MS}_{g,1}^k$ are precisely the metric spines of half-plane differentials in $\mathcal{H}_k(\Sigma, p)$, as $(\Sigma,p)$ varies over the Teichm\"{u}ller space $\mathcal{T}_{g,1}$, an observation which we shall use below.\\

As a consequence of Proposition \ref{mex} we also have:

\begin{lem}\label{mgk}
The space of metric graphs $\mathcal{MS}_{g,1}^0$ is homeomorphic to $\mathbb{R}^{6g-6+3}$. For $k\geq 1$, the space of metric graphs $\mathcal{MS}_{g,1}^k$ is homeomorphic to $\mathbb{R}^{6g-6 +k+ 2} \times S^1$.
\end{lem}
\begin{proof}

The first statement concerns the space of ribbon graphs (with a punctured-disk complement) that have been parameterized elsewhere (see for example \cite{PennDec}, \cite{MulPenk}). In fact, Strebel's theorem (as in the Introduction) provides the isomorphism $\mathcal{MS}_{g,1}^0 \cong \mathcal{T}_{g,1} \times \mathbb{R}_+ \cong \mathbb{R}^{6g-6+3}$.

To obtain the second statement, we need to add in $k$ infinite edges to such a ribbon graph in $\mathcal{MS}_{g,1}^0$. The simplest way to do this is to add in an extra vertex $v_0$ on the ribbon graph (viewed as the boundary of a disk punctured at $p$) such that the additional $k$ edges emerge from $v_0$. There is then an $S^1$-parameter space of choices for placing the vertex $v_0$.  Any other graph in $\mathcal{MS}_{g,1}^k$ is obtained by taking a metric expansion at $v_0$, in the sense of Definition \ref{mexp} (see Figure 12). Since the degree of $v_0$ is $k+2$, we see by Proposition~\ref{mex} that the total space $\mathcal{MS}_{g,1}^k$ then is homeomorphic to $\mathcal{MS}_{g,1}^0 \times S^1\times \mathbb{R}^{k -1} \cong \mathbb{R}^{6g-6 + k +2} \times S^1$ as required.   \end{proof}

\textit{Remark.} Alternatively, for the second statement, given a ribbon graph in $\mathcal{MS}_{g,1}^0$, one could replace the punctured disk in its complement by a $k$-planar end with a choice of rotation, and hence the total space $\mathcal{MS}_{g,1}^k$ would also be seen from this perspective as homeomorphic to $\mathcal{MS}_{g,1}^0 \times \P(k) \times S^1 \cong \mathbb{R}^{6g-6 + k +2} \times S^1$ (where we have invoked Proposition~\ref{pk}). 

\begin{figure}
  \centering
  \includegraphics[scale=0.55]{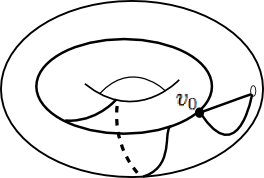}\\
  \caption{A metric spine in $\mathcal{MS}_{1,1}^2$ before ``expansion" as in the proof of Lemma \ref{mgk}. }
\end{figure}

\begin{proof}[Proof of Corollary \ref{thm3}]
We begin with the bundle $\mathcal{Q}_{g,1}^{k+2}$ of meromorphic quadratic differentials  with a pole of order $(k+2)$ over the \tec space $\T_{g,1}$. In that total space $\mathcal{Q}_{g,1}^{k+2}$ of meromorphic quadratic differentials  with a pole of order $(k+2)$,
consider the subset $\mathcal{HP}_{k,g}$ consisting of half-plane differentials,, i.e.
\begin{equation*}
\mathcal{HP}_{k,g} = \bigcup\limits_{(\Sigma,p)} \mathcal{HP}_k(\Sigma,p) \subset \mathcal{Q}_{g,1}^{k+2}
\end{equation*}
where $(\Sigma,p)$ varies over $\mathcal{T}_{g,1}$.

By Theorem~\ref{main}, we know that each  $\mathcal{HP}_k(\Sigma,p) \cong  \mathbb{R}^{k} \times S^1$ and hence 
\begin{center}
$\mathcal{HP}_{k,g}  \cong \mathcal{T}_{g,1}  \times  \mathbb{R}^{k} \times S^1$.
\end{center}
It remains to show that the map $\Phi: \mathcal{HP}_{k,g}\to \mathcal{MS}^k_{g,1}$ that assigns the metric spine of the corresponding half-plane differential, is a homeomorphism. 

Any graph in  $\mathcal{MS}_{g,1}^k$  corresponds to a \textit{unique} pointed Riemann surface $(\Sigma,p)$ and half-plane differential in $\mathcal{HP}_k(\Sigma,p)$ by attaching to the spine $k$ Euclidean half-planes between the $k$ infinite-length edges incident at the puncture. (The differential is then the one induced by the standard differential $d\zeta^2$ in the usual complex coordinate $\zeta$ on each half-plane - see \S5.1 for a related discussion.)  It is easy to see that this assignment forms the inverse for the map $\Phi$.

Hence the map $\Phi$ is a bijection. Moreover, the inverse map constructed above is clearly continuous. Since the \tec space $\mathcal{T}_{g,1}$ is homeomorphic to $\mathbb{R}^{6g-6+2}$, applying  Invariance of Domain to the lift of $\Phi$ to a map between the universal covers  $\widetilde{\mathcal{HP}_{k,g}}$ and $\widetilde{\mathcal{MS}^k_{g,1}}$ (both homeomorphic to $\mathbb{R}^{6g-6 + k+3}$) completes the proof.  \end{proof}

\bibliographystyle{amsalpha}
\bibliography{qdref}

\end{document}